\documentclass[]{amsart}
\usepackage{amsfonts}
\usepackage{amsmath}
\usepackage{amssymb}
\usepackage{mathrsfs}  
\usepackage{mathtools}
\usepackage{graphicx, float}
\usepackage[hmargin=2.8cm,vmargin=2.3cm]{geometry}
\usepackage{enumitem}
\usepackage[dvipsnames]{xcolor}

\usepackage{tikz-cd}
\usepackage{siunitx}

\graphicspath{{Figures/}}

\newtheorem{theorem}{Theorem}
\newtheorem{proposition}[theorem]{Proposition}

\newtheorem{corollary}[theorem]{Corollary}
\newtheorem{lemma}[theorem]{Lemma}

\theoremstyle{definition}
\newtheorem{definition}[theorem]{Definition}
\newtheorem{example}[theorem]{Example}

\newtheorem{remark}[theorem]{Remark} 

\numberwithin{equation}{section}
\numberwithin{theorem}{section}

\newenvironment{introtheorem}[1]
  {\begingroup\renewcommand{\thetheorem}{#1}\begin{theorem}}
  {\end{theorem}\endgroup\addtocounter{theorem}{-1}}

\DeclareMathOperator{\Id}{Id}
\DeclareMathOperator{\id}{id}

\DeclareMathOperator{\grad}{grad}
\DeclareMathOperator{\sgrad}{sgrad}
\DeclareMathOperator{\ad}{ad}
\DeclareMathOperator{\Ad}{Ad}

\DeclareMathOperator{\SO}{SO}

\DeclareMathOperator{\GL}{GL}
\DeclareMathOperator{\PSL}{PSL_2(\mathbb R)}
\DeclareMathOperator{\RP}{\mathbb{RP}^1}
\DeclareMathOperator{\End}{End}
\DeclareMathOperator{\tr}{tr}
\DeclareMathOperator{\Hom}{Hom}
\DeclareMathOperator{\Skew}{skew}

\DeclareMathOperator{\so}{\mathfrak{so}}
\DeclareMathOperator{\spl}{\mathfrak{sl}_2(\mathbb R)}
\DeclareMathOperator{\Der}{\mathfrak{D}}
\DeclareMathOperator{\bor}{\mathfrak{b}}

\DeclareMathOperator{\src}{src}
\DeclareMathOperator{\trg}{trg}

\DeclareMathOperator{\Ker}{Ker}
\DeclareMathOperator{\Diff}{Diff}
\DeclareMathOperator{\Dens}{Dens}
\DeclareMathOperator{\SDiff}{SDiff}

\DeclareMathOperator{\SVect}{SVect}

\DeclareMathOperator{\Vect}{Vect}

\DeclareMathOperator{\rank}{rank}

\newcommand{\torus}{\mathbb{T}}
\newcommand{\VecttorC}{\Vecttor_{\mbb C}}
\newcommand{\SVecttor}{\SVect(\torus)}

\newcommand{\Vecttor}{\Vect(\torus)}
\newcommand{\SDifftor}{\SDiff(\torus)}
\newcommand{\Difftor}{\Diff(\torus)}
\newcommand{\Denstor}{\Dens(\torus)}
\newcommand{\Base}{B}
\newcommand{\Groid}{\mathcal{G}} 
\newcommand{\Aroid}{\mathcal{A}}
\newcommand{\anchor}{\#}
\newcommand{\vertbundle}{\Ker\anchor}
\newcommand{\p}{\partial}

\newcommand{\mf}{\mathfrak}
\newcommand{\mc}{\mathcal}
\newcommand{\mbb}{\mathbb}
\newcommand{\mbf}{\mathbf}
\newcommand{\norm}[1]{\left|#1\right|}
\newcommand{\ip}[2]{\langle #1 , #2 \rangle}
\newcommand{\set}[1]{\left\{#1 \right\}}

\newcommand{\ChEnd}{\boldsymbol{\Gamma}}
\newcommand{\II}{\mathrm{II}}
\newcommand{\madj}{\dagger}
\newcommand{\Lie}{\mathcal{L}}
\newcommand{\trans}{\intercal}
\newcommand{
\RSLA}{Riemannian submersion Lie algebroid}
\newcommand{\ecc}{\varepsilon}
\newcommand{\curv}{\mathcal{R}}

\title{Curvature Decomposition for Riemannian Lie Algebroids}
\author{
René Langøen 
}

\address{R.L.: Department of Mathematics, University of Bergen, Norway.}
\email{rene.langoen@uib.no}

\usepackage[pdfencoding=auto, psdextra, hypertexnames=false]{hyperref}
\hypersetup{
    colorlinks=false,       
    linkcolor=red,          
    citecolor=green,        
    filecolor=red,      
    urlcolor=blue           
}
\usepackage[ ]{amsrefs}
\begin{document}

\begin{abstract}
Classical curvature formulas for Lie groups and Riemannian submersions are recast in terms of Lie groupoids with right-invariant source-fibre metrics and Riemannian Lie algebroids. We derive a sectional-curvature formula for such groupoids extending the Arnold-Milnor ``1-2-3-4'' formula. For transitive Riemannian Lie algebroids, we factor the base and vertical Levi-Civita connections through Lie algebroids of metric derivations and obtain Lie-algebroid analogues of O'Neill's curvature formulas. Applied to the action of the diffeomorphism group of the torus on densities, these formulas recover Arnold's curvature of the measure-preserving diffeomorphism group and give explicit Wasserstein sectional curvatures at the uniform density. Finite-dimensional examples include the rotational action algebroid of the round sphere and the WGS84 ellipsoid, and a smooth family of negative curvatures on non-abelian isotropy groups over the real projective line.
\end{abstract}

\maketitle

\tableofcontents

\section{Introduction}

The aim of the present work is to revisit several classical curvature constructions---Riemannian geometry on Lie groups with invariant metrics, principal bundles with metric-compatible connections, and O'Neill's theory of Riemannian submersions---and to reinterpret these in the language of Lie algebroids.
We extend the Arnold-Milnor sectional-curvature formula to Lie groupoids
with right-invariant source-fibre metrics.
For the action of diffeomorphisms of the torus on densities, we recover
Arnold's curvature formula on the isotropy groups and obtain explicit
Wasserstein curvature formulas on the base from the same algebroid Levi-Civita connection.

A transitive Lie algebroid combines infinitesimal directions on the base manifold with directions coming from isotropy, through the short exact sequence
\[
\begin{tikzcd}
\vertbundle \ar[r, "\iota", hook] &
\Aroid \ar[r, "\anchor", two heads] &
T\Base.
\end{tikzcd}
\]
A metric on $\Aroid$ determines an orthogonal splitting of this sequence as vector bundles. Its curvature measures the failure of the horizontal lift to preserve the Lie algebroid bracket. More generally, we use curvature as the obstruction to an anchor-preserving map being a Lie algebroid morphism.

Atiyah's work on connections in principal bundles led to the exact sequence now associated with the Atiyah algebroid \cite{atiyah_complex_1957}; the groupoid and algebroid constructions used here follow Mackenzie \cites{mackenzie_lie_1987,mackenzie_general_2005}. Crainic and Fernandes established the integrability criteria for Lie algebroids \cites{crainic_integrability_2003,crainic_integrability_2004}. Fernandes' treatment of Lie algebroid connections \cite{fernandes_lie_2002} has been an important conceptual influence on the present work. We consider right-invariant metrics on source fibres, rather than the metric structures studied by del Hoyo and Fernandes in their theory of Riemannian groupoids \cites{del_hoyo_riemannian_2018,del_hoyo_riemannian_2019}.

Arnold's interpretation of hydrodynamics as geodesic motion on diffeomorphism groups \cite{arnold_sur_1966} and Milnor's study of invariant metrics \cite{milnor_curvatures_1976} provide the Lie-group curvature formulas used below. Izosimov and Khesin extended the hydrodynamical setting to groupoids with right-invariant source-fibre metrics, treating vortex sheets and multiphase fluids \cites{izosimov_vortex_2018,izosimov_geometry_2023}. Related constructions in optimal transport and geometric mechanics are discussed in \cite{modin_geometry_2017}.

Our principal application is the action of the diffeomorphism group of the torus on densities. Its isotropy groups consist of measure-preserving diffeomorphisms, while the base carries the Wasserstein metric. The vertical identity recovers Arnold's curvature formula. The horizontal identity gives explicit Fourier formulas for the curvature of the horizontal lift and the Wasserstein sectional curvatures at the uniform density, specialising the curvature calculations of Otto and Lott~\cites{otto_geometry_2001,lott_geometric_2007}. Both computations proceed from the same algebroid Levi-Civita connection.

A central result is a sectional-curvature formula for Lie groupoids endowed with right-invariant source-fibre metrics. It extends the classical Arnold-Milnor $1$-$2$-$3$-$4$ formula from Lie groups to Lie groupoids.
\begin{introtheorem}{\ref{th:2-1-2-3-4}}
The sectional curvature of a Lie groupoid $\Groid$ with a right-invariant source-fibre metric, in the direction determined by an orthonormal pair $\tilde X,\tilde Y\in T_g^{\src}\Groid$, is
\[
C^{\Groid}(\tilde X,\tilde Y)
=2\bigl(\anchor X\cdot\ip{\alpha}{Y}-\anchor Y\cdot\ip{\alpha}{X}\bigr)
+\ip{\delta}{\delta}+2\ip{\alpha}{\beta}-3\ip{\alpha}{\alpha}-4\ip{B_X}{B_Y},
\]
where $X,Y$ are local orthonormal sections of $\Aroid$ near $y=\trg(g)$, extending $X_y=dR_{g^{-1}}\tilde X$ and $Y_y=dR_{g^{-1}}\tilde Y$, and
\[
2\alpha=[X,Y]_{\Aroid}=\ad_XY,\qquad
2\delta=\ad^\madj_XY+\ad^\madj_YX,
\]
\[
2\beta=\ad^\madj_XY-\ad^\madj_YX,\qquad
2B_X=\ad^\madj_XX,\qquad 2B_Y=\ad^\madj_YY.
\]
The right-hand side is evaluated at $y$ and is independent of the chosen local orthonormal extensions.
\end{introtheorem}
Compared with the Lie-group formula, the expression contains additional derivatives of the bracket coefficients along the anchor. These terms account for the variation of the algebroid structure over the base and vanish when the bracket coefficients in the chosen orthonormal frame are constant. The formula thus places Lie groups, principal bundles, Riemannian manifolds, and more general Lie groupoids into a common curvature framework.

Boucetta states O'Neill's sectional-curvature relations for Riemannian
Lie algebroids~\cite{boucetta_riemannian_2011}.
Our contribution is a factorization of the full horizontal and vertical curvature tensors through the Lie algebroids of metric derivations, together with explicit formulas for the curvatures of the diagonal projections $g,\widehat g$. These formulas separate the contributions used in the applications below. For a Riemannian submersion Lie algebroid, the factorization takes the following form.
\begin{introtheorem}{\ref{th:algebroid.ONeill}}
Let $\Aroid$ be a \RSLA{}, with orthogonal horizontal lift $\gamma:T\Base\to\Aroid$, vertical projection $\omega:\Aroid\to\vertbundle$, and induced Levi-Civita connections $\nabla$, $\widetilde\nabla$, and $\widehat\nabla$ on $\Aroid$, $T\Base$, and $\vertbundle$, respectively.
\[
\begin{tikzcd}[column sep=large,row sep=large]
\Der_{\so}(\vertbundle)
& \Der_{\so}(\Aroid) \ar[l,"\widehat g",swap] \ar[r,"g"]
& \Der_{\so}(T\Base) \\
\vertbundle \ar[r,"\iota",hook] \ar[u,"\widehat\nabla"]
& \Aroid \ar[r,"\anchor",two heads] \ar[u,"\nabla"]
  \ar[l,"\omega",bend left=25]
& T\Base \ar[l,"\gamma",bend left=25] \ar[u,"\widetilde\nabla"]
\end{tikzcd}
\]
Then
\[
\curv^{\widetilde\nabla}
=(\nabla\circ\gamma)^*\curv^g
+g\circ\gamma^*\curv^\nabla
+g\circ\nabla\circ\curv^\gamma,
\]
\[
\curv^{\widehat\nabla}
=(\nabla\circ\iota)^*\curv^{\widehat g}
+\widehat g\circ\iota^*\curv^\nabla,
\]
where $g(D)=\anchor\circ D\circ\gamma$ and $\widehat g(D)=\omega\circ D\circ\iota$ are the diagonal projections of metric derivations, and $\curv^\varphi$ denotes the curvature of an anchor-preserving map $\varphi$.
\end{introtheorem}
These identities separate the contribution of the source-fibre curvature from those of the diagonal projections and the horizontal splitting. In the torus-density application, the Levi-Civita connection $\nabla$ is flat, so the curvature of the isotropy groups and that of the Wasserstein base are determined by the remaining terms.

Section~\ref{sec:preliminaries} fixes conventions and recalls the required groupoid, algebroid, and connection constructions. Sections~\ref{sec:riemannian.struc.aroids} and~\ref{sec:Riemannian.submersion.Lie.algebroids} establish the sectional-curvature and curvature-decomposition formulas. The finite-dimensional applications in Section~\ref{sec:examples} treat the round sphere, its nonconstant-metric Earth-ellipsoid counterpart, and a family of negative isotropy curvatures over the projective line. Section~\ref{ex:diffeo.groupoid} contains the principal torus-density application. The proofs of the main curvature identities are collected in Section~\ref{sec:proofs.of.main.theorems}.

\section{Preliminaries and conventions}\label{sec:preliminaries}

\subsection{Invariant metrics and curvature conventions}\label{sec:Lie.groups}
We fix the conventions for invariant metrics and curvature used throughout.  For a Lie group $G$ with Lie algebra $\mf g=T_{\Id}G$, we identify $X\in\mf g$ with the right-invariant vector field
\[
\widetilde X_g=dR_gX, \qquad \text{for } X\in\mf g,\ g\in G.
\]
The bracket on $\mf g$ is induced by this identification,
\[
[X,Y]_{\mf g}=[\widetilde X,\widetilde Y]_{TG}(\Id),
\qquad \text{for }X,Y\in\mf g.
\]
Thus, for a matrix Lie group, our bracket is the negative commutator:
\[
[A,B]_{\mf g}=-(AB-BA),
\qquad \text{for }A,B\in\mf g\subset M_n(\mbb R).
\]

An inner product $\ip{\cdot}{\cdot}_{\mf g}$ determines a right-invariant metric on $G$ by
\[
\ip{\widetilde{X}}{\widetilde{Y}}_G(g)=
\ip{dR_{g^{-1}}\widetilde{X}}{dR_{g^{-1}}\widetilde{Y}}_{\mf g},
\qquad \text{for }\widetilde{X},\widetilde{Y} \in T_gG,\ g\in G.
\]
For right-invariant vector fields, the Koszul formula consequently reduces to
\[
2\ip{\nabla_XY}{Z}_{\mf g}
=\ip{[X,Y]}{Z}_{\mf g}
-\ip{Y}{[X,Z]}_{\mf g}
-\ip{X}{[Y,Z]}_{\mf g},
\qquad \text{for }X,Y,Z\in\mf g,
\]
or, equivalently,
\[
\nabla_XY=\frac12\bigl(\ad_XY-\ad_X^{\trans}Y-\ad_Y^{\trans}X\bigr),
\qquad \text{for }X,Y\in\mf g,
\]
where $\ad_X^{\trans}$ denotes the transpose with respect to $\ip{\cdot}{\cdot}_{\mf g}$.  In particular, the Riemannian geometry of a right-invariant metric is encoded by the metric Lie algebra; see \cite{milnor_curvatures_1976}.

Our curvature convention is
\[
\curv(X,Y)Z=\nabla_{[X,Y]}Z-\nabla_X\nabla_YZ+\nabla_Y\nabla_XZ,
\qquad \text{for }X,Y,Z\in\mf g,
\]
and, for an orthonormal pair $X,Y\in\mf g$, the sectional curvature is $C(X,Y)=\ip{\curv(X,Y)X}{Y}$.  In these conventions, Arnold's formula \cite{arnold_sur_1966} takes the following form.
\begin{theorem}[\cite{arnold_sur_1966}]\label{th:arnold.1-2-3-4}
The sectional curvature of a Lie group $G$ with a right-invariant metric, in the direction determined by an orthonormal pair of vectors $X,Y$ in the Lie algebra $\mf g$ is given by the formula
\[
C({X}, {Y}) = \ip{\delta}{\delta} + 2\ip{\alpha}{\beta} - 3\ip{\alpha}{\alpha} -4\ip{B_X}{B_Y},
\]
where  
\[
2\alpha = [X, Y] = \ad_X Y, \quad 2\delta = \ad^{\trans}_X Y + \ad^\trans_Y X,
\]
\[
2\beta = \ad^{\trans}_Y X - \ad^{\trans}_X Y, \quad 
2B_X = \ad^\trans_X X, \quad 2B_Y= \ad^\trans_Y Y.
\]

\end{theorem}

\subsection{Riemannian submersions}\label{sec:oneill}
\begin{definition}\label{def:riemannian.submersion}
A smooth map $\pi:N\to\Base$ between Riemannian manifolds is a \emph{Riemannian submersion} if it is a surjective submersion and, for every $x\in N$, the restriction
\[
d\pi_x:(\Ker d\pi_x)^{\perp}\longrightarrow T_{\pi(x)}\Base
\]
is an isometry.
\end{definition}
Set $\mc V=\Ker d\pi$ and $\mc H=\mc V^{\perp}$.  Thus
\[
TN=\mc H\oplus_{\perp}\mc V,
\qquad
\begin{tikzcd}
\mc V \ar[r, hook] & TN \ar[r, "d\pi", two heads] & \pi^*T\Base,
\end{tikzcd}
\]
and $d\pi|_{\mc H}$ is an isometry.  Let $(\cdot)^{\mc H}:TN\to\mc H$ and $(\cdot)^{\mc V}:TN\to\mc V$ denote the horizontal and vertical projections, respectively.  Thus every $X\in\Gamma TN$ decomposes as $X=X^{\mc H}+X^{\mc V}$.  If $\nabla$ denotes the Levi-Civita connection of $N$, O'Neill's tensors are, for $X,Y\in\Gamma TN$,
\[
T_XY=(\nabla_{X^{\mc V}}Y^{\mc V})^{\mc H}
      +(\nabla_{X^{\mc V}}Y^{\mc H})^{\mc V},
\]
\[
A_XY=(\nabla_{X^{\mc H}}Y^{\mc V})^{\mc H}
      +(\nabla_{X^{\mc H}}Y^{\mc H})^{\mc V}.
\]
For vertical vector fields $X^{\mc V},Y^{\mc V}\in\Gamma\mc V$, the tensor
\[
T_{X^{\mc V}}Y^{\mc V}=(\nabla_{X^{\mc V}}Y^{\mc V})^{\mc H}
\]
is the second fundamental form of the fibres.  For horizontal vector fields $X^{\mc H},Y^{\mc H}\in\Gamma\mc H$,
\[
A_{X^{\mc H}}Y^{\mc H}
=(\nabla_{X^{\mc H}}Y^{\mc H})^{\mc V}
=\frac12[X^{\mc H},Y^{\mc H}]^{\mc V}
\]
measures the non-integrability of $\mc H$.  Let $\widetilde\nabla$ and $\widehat\nabla$ denote the Levi-Civita connections of the base and the fibres, with respective curvature tensors $\widetilde{\curv}$ and $\widehat{\curv}$.  The identities needed below are the following parts of O'Neill's formulas \cite{oneill_fundamental_1966}.
\begin{theorem}[\cite{oneill_fundamental_1966}]\label{th:oneill.formulas.riemannian.submersion}
Let $V_1, V_2, V_3, V_4$ be vertical vector fields on $N$, then
\[
\ip{\curv (V_1, V_2)V_3} {V_4} = \ip{\widehat{\curv}(V_1,V_2)V_3 }{V_4} - \ip{T_{V_1} V_3}{T_{V_2}V_4} + \ip{T_{V_2} V_3}{T_{V_1} V_4}.
\]
Let $H_1, H_2, H_3, H_4$ be horizontal vector fields on $N$, then
\begin{multline*}
\ip{\curv(H_1,H_2) H_3}{H_4} = \ip{\tilde{\curv}(H_1,H_2) H_3 }{H_4} - 2\ip{A_{H_1} H_2}{A_{H_3} H_4} \\
+ \ip{A_{H_2} H_3}{A_{H_1} H_4} + \ip{A_{H_3} H_1}{A_{H_2} H_4} .
\end{multline*}
\end{theorem}
The vertical identity is the Gauss equation for the fibres, while the horizontal correction terms record the failure of $\mc H$ to be involutive.  These are the two features that reappear in the Lie-algebroid formulation.

\subsection{Lie groupoids}\label{sec:Lie.groupoids}
We recall the definitions and examples of Lie groupoids and Lie algebroids used below, following \cite{mackenzie_general_2005}.

\begin{definition}
A \emph{Lie groupoid} $\Groid\rightrightarrows\Base$ consists of smooth manifolds $\Groid$ and $\Base$, surjective submersions $\src,\trg:\Groid\to\Base$, smooth unit and inversion maps $\id:\Base\to\Groid$ and $g\mapsto g^{-1}$, and a smooth multiplication
\[
\Groid*\Groid \longrightarrow\Groid,
\qquad
(h,g)\longmapsto hg,
\qquad
\Groid*\Groid=\set{(h,g)\in\Groid\times\Groid\mid \src(h)=\trg(g)},
\]
satisfying the groupoid associativity, unit, and inverse identities.  
For $x,y\in\Base$, set
\[
\Groid_x=\src^{-1}(x),
\qquad
\Groid^y=\trg^{-1}(y),
\qquad
\Groid_x^y=\Groid_x\cap\Groid^y.
\]
The set $\Groid_x$ is called the \emph{source fibre} over $x$, and $\Groid^y$ is called the \emph{target fibre} over $y$.  The intersection $\Groid_x^x$ is the \emph{isotropy group} at $x\in\Base$ and is denoted by $H_x=\Groid_x^x$.
\end{definition}

\begin{example}\label{ex:lie.groupoids}
Let $\Base$ be a smooth manifold, and $G$ a Lie group.
\begin{enumerate}
\item Let $G\rightrightarrows\set{*}$ be a Lie group over a point.  This is a Lie groupoid with $\src=\trg=*$. 

\item The Cartesian product $\Base\times\Base\rightrightarrows\Base$ is a Lie groupoid called the \emph{pair groupoid}, with $\src(y,x)=x$, $\trg(y,x)=y$, and product $(z,y)(y,x)\coloneq(z,x)$, for $x,y,z\in\Base$.  The object inclusion is the diagonal $\id_x=(x,x)$, and the inverse of $(y,x)$ is $(x,y)$.

\item The product manifold $\Base\times G\times\Base$ is the \emph{trivial Lie groupoid} on $\Base$ with group $G$, with structure maps
\[
\src(y,g,x)=x,
\qquad
\trg(y,g,x)=y,
\qquad
\id_x=(x,\Id,x),
\]
and
\[
(z,h,y)(y,g,x)=(z,hg,x),
\qquad
(y,g,x)^{-1}=(x,g^{-1},y),
\]
for $x,y,z\in\Base$ and $g,h\in G$.  A Lie groupoid for which $(\src,\trg):\Groid\to\Base\times\Base$ is a surjective submersion is called \emph{locally trivial}; locally it is modelled on a trivial Lie groupoid.

\item\label{ex:action.groupoid} Let $G\times\Base\to\Base$ be a smooth action.  The product manifold $G\times\Base$ is the \emph{action Lie groupoid}, denoted by $G\ltimes\Base\rightrightarrows\Base$, with
\[
\src(g,x)=x,
\qquad
\trg(g,x)=gx,
\qquad
(g,x)^{-1}=(g^{-1},gx),
\qquad
(h,gx)(g,x)=(hg,x),
\]
for $g,h\in G$ and $x\in\Base$, and object inclusion $\id_x=(\Id,x)$.  There is an injective morphism into the trivial Lie groupoid $\Base\times G\times\Base$, given by $(g,x)\mapsto(gx,g,x)$.
\end{enumerate}
\end{example}

\begin{example}[The Gauge Groupoid of a Principal Bundle]\label{ex:gauge.groupoid}
Let $P(\Base,H,\pi)$ be a principal $H$-bundle.  Consider the diagonal right action of $H$ on $P\times P$, given by
\[
(p_2,p_1)\cdot h=(p_2\cdot h,p_1\cdot h),
\qquad (p_2,p_1)\in P\times P,\ h\in H,
\]
and denote the orbit of $(p_2,p_1)$ by $[p_2,p_1]$.  The quotient $(P\times P)/H$ is a Lie groupoid over $\Base$, with source and target maps
\[
\src([p_2,p_1])=\pi(p_1),
\qquad
\trg([p_2,p_1])=\pi(p_2),
\]
for $p_1,p_2\in P$.  The object inclusion is $\id_x=[p,p]$, for any $p\in\pi^{-1}(x)$, and the product is given by
\[
[p_3\cdot h,p_2\cdot h][p_2,p_1]=[p_3,p_1],
\qquad p_1,p_2,p_3\in P,\ h\in H.
\]
Finally, $[p_2,p_1]^{-1}=[p_1,p_2]$.  This is the \emph{gauge groupoid associated to the principal bundle $P(\Base,H,\pi)$}.

Conversely, let $\Groid\rightrightarrows\Base$ be a Lie groupoid for which $(\src,\trg):\Groid\to\Base\times\Base$ is a surjective submersion, and fix $x\in\Base$.  Then
\[
P=\Groid_x=\src^{-1}(x),
\qquad
H_x=\Groid_x^x,
\qquad
\pi=\trg_x:P\to\Base,
\]
defines a principal $H_x$-bundle.  Its right action is the groupoid multiplication $(p,h)\mapsto ph$, for $p\in\Groid_x$ and $h\in H_x$, and is free and transitive on the fibres of $\trg_x$; see \cite[Section~1.3]{mackenzie_general_2005}.

Thus a principal bundle is ``equivalent'' to a Lie groupoid for which $(\src,\trg)$ is a surjective submersion.  However, the correspondence is not canonical, since passing from the Lie groupoid to the principal bundle requires the choice of a base point $x\in\Base$.  Phenomena in principal bundle theory that depend on a reference point can often be formulated intrinsically in groupoid terms, independently of such a choice.  Examples include the holonomy groups arising from a principal connection and the frame bundle $\text{Fr}(E)$ of a vector bundle $E$, considered next.
\end{example}

\begin{example}[General Linear Groupoid]\label{ex.GL(E)}
For a rank-$n$ vector bundle $E\to\Base$, the \emph{general linear groupoid} $\GL(E)\rightrightarrows\Base$ is given by
\[
\GL(E)_x^y=\set{\phi:E_x\to E_y\mid \phi\text{ is a linear isomorphism}},
\qquad x,y\in\Base,
\]
with multiplication given by composition.  It is canonically isomorphic to the gauge groupoid of the frame bundle $\text{Fr}(E)$.  The frame bundle is constructed by comparing each fibre $E_x$ with a fixed model vector space $\mbb R^n$, and therefore depends on this choice.  In contrast, the general linear groupoid is defined intrinsically, without the need to choose a reference model fibre.
\end{example}

\subsection{Lie algebroids}
The Lie algebroid of a Lie groupoid is constructed from right-invariant vector fields tangent to the source fibres. We first define abstract Lie algebroids, then recall this construction and the examples used below.

\begin{definition}\label{def:lie.algebroid}
A Lie algebroid $\Aroid$ over a manifold $\Base$, is a vector bundle $\Aroid \to \Base$, with a vector bundle map $\anchor:\Aroid \to T\Base$ over $\Base$, called the \emph{anchor} of $\Aroid$, and a bracket $[ \ , \ ]:\Gamma\Aroid \times \Gamma \Aroid \to \Gamma \Aroid$ which is $\mathbb{R}$-bilinear, alternating and satisfies the Jacobi-identity, subject to the conditions that
\begin{align}
[X, fY]_\Aroid = f[X, Y]_\Aroid + \anchor X (f) Y, \label{eq:bracket.deriv}  \\
\anchor [X, Y]_\Aroid = \left[ \anchor X , \anchor Y  \right]_{T\Base},\label{eq:anchor.morphism}
\end{align}
for all $X,Y \in \Gamma \Aroid, f\in C^{\infty}(\Base)$. 
The Lie algebroid is \emph{transitive} if $\anchor$ is fibrewise surjective. In this case, the kernel of $\anchor:\Aroid \to T\Base$ is a bundle of Lie algebras over $\Base$, and we call $\vertbundle$ the \emph{vertical bundle} in $\Aroid$, see \cite{mackenzie_general_2005} for more details. 
\end{definition}

We now define morphisms of Lie algebroids over a common base manifold. 
\begin{definition}\label{def:lie.algebroid.morphism}
Let $\Aroid_1$ and $\Aroid_2$ be two Lie algebroids over the same base space $\Base$. Then a \emph{morphism of Lie algebroids} $\varphi: \Aroid_1 \to \Aroid_2$ over $\Base$, is a vector bundle morphism such that $\anchor_2\circ \varphi = \anchor_1$ and $\varphi[X, Y]_1 = [\varphi(X), \varphi(Y)]_2$, for all $X, Y \in \Gamma\Aroid_1$. 
\end{definition}
\begin{remark}
A \emph{transitive} Lie algebroid gives a short exact sequence of Lie algebroids
\begin{equation}\label{eq:trans.lie.abroid}
\begin{tikzcd}
\vertbundle \ar[r, "\iota", hook] & \Aroid \ar[r, "\anchor", two heads]  &T\Base,
\end{tikzcd}
\end{equation}
Here $\iota$ and $\anchor$ are morphisms of Lie algebroids.  The tangent bundle $T\Base$ has identity anchor, whereas $\vertbundle$ is a \emph{totally intransitive Lie algebroid}, with zero anchor.  This sequence is the Lie-algebroid counterpart of the horizontal-vertical decomposition associated to a Riemannian submersion.  The latter lives over $N$ and therefore involves the pullback $\pi^*T\Base$, whereas the algebroid sequence lives entirely over $\Base$, even when $\rank(\Aroid)\geq\dim(\Base)$.
\end{remark}

For a Lie group $G$, one constructs its Lie algebra as the space of right-invariant vector fields on $G$, with the Lie algebra bracket given by the restriction of the bracket of vector fields, see Section~\ref{sec:Lie.groups}. Similarly, given a Lie groupoid $\Groid \rightrightarrows \Base$, we want to consider right-invariant vector fields on $\Groid$. Let $g\in \Groid$, with $\src(g) = x$ and $\trg(g) = y$. Define right multiplication on $\Groid_y $ by 
\[
R_g : \Groid_y \to \Groid_x, \quad h\mapsto hg.
\]
Note that $R_g$ cannot be defined on the entire $\Groid$, only on the fibre $\Groid_y = \src^{-1}(y)$. This implies that the notion of right-invariance only makes sense for vector fields tangent to source fibres. 
\begin{definition}\label{def:right.invariant.vector.fields.groupoid}
A \emph{right-invariant} vector field on a Lie groupoid $\Groid$ is a vector field $\tilde{X}\in \Gamma T\Groid$ satisfying
\begin{enumerate}
\item $\tilde{X}$ is tangent to source fibres, i.e. $d\src \tilde{X} = 0$,
\item $\tilde{X}_{gh} 
= dR_h\tilde{X}_{g} $, where $\src(g) = \trg(h)$. 
\end{enumerate}
\end{definition}
\begin{figure}
\centering
\includegraphics[width=0.75\textwidth]{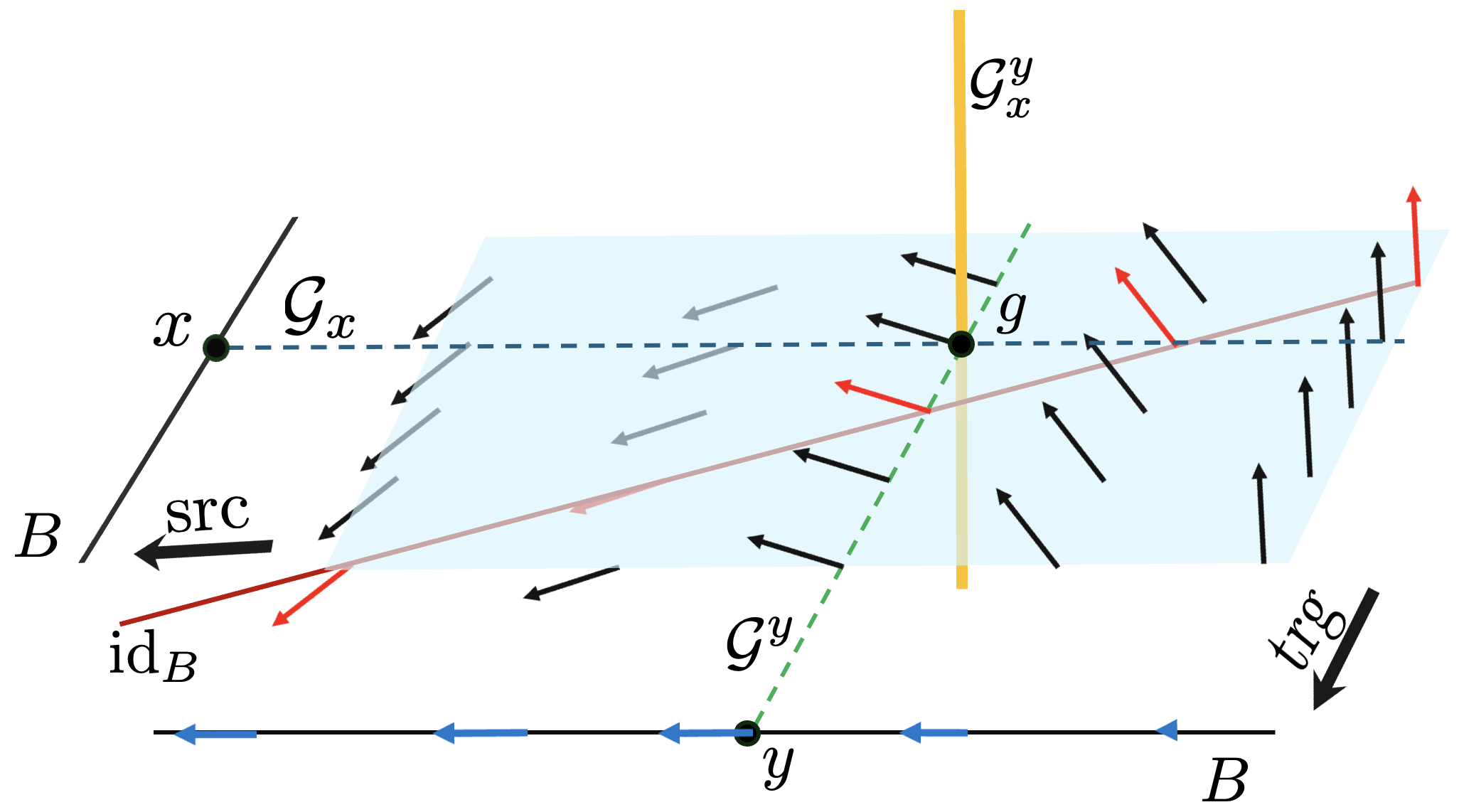}
\caption{A right-invariant vector field tangent to the source fibres of $\Groid$. Its values are related by right translations along the target fibres, and its restriction to $\id_{\Base}$ is a section of the Lie algebroid $\Aroid$. The anchor is induced by $d\trg$. When $(\src,\trg):\Groid \to \Base \times \Base$ is a surjective submersion, each source fibre $\Groid_x$ is a principal bundle with projection $\trg$.}
\label{fig:groupoid.vectorfield}
\end{figure}
Evidently, the value of a right-invariant vector field $\tilde{X}$ on $\Groid$ is determined by its values at the identity submanifold $\id_\Base $.  Conversely, consider the vector bundle 
\[
\Aroid \coloneq \bigcup_{x\in \Base} T_{\id_x} \Groid_x \longrightarrow \Base,
\]
and given a section $X:\Base \to \Aroid$. Then $X$ uniquely determines a right-invariant vector field $\tilde X$ on $\Groid$ by 
\begin{equation}\label{eq:right.invariant.exstension}
\tilde{X}_g =  dR_g{X}_{y}, \quad \text{where } y = \trg(g).
\end{equation}
This correspondence leads to the following Definition/Theorem, see also Figure \ref{fig:groupoid.vectorfield}. 
\begin{theorem}[Section~3.5 in \cite{mackenzie_general_2005}]\label{th:lie.algebroid.of.lie.groupoid}
The \emph{Lie algebroid of a Lie groupoid} $\Groid \rightrightarrows \Base$ is the vector bundle 
\[
\Aroid \coloneq \bigcup_{x\in \Base} T_{\id_x} \Groid_x \longrightarrow \Base,
\]
with anchor map $\anchor:\Aroid \to T\Base$,
\[
\anchor_x:\Aroid_x \to T_x\Base, \quad X \mapsto d\trg_{\id_x} X,
\]
and Lie algebroid bracket $[\cdot, \cdot]_{\Aroid}:\Gamma\Aroid \times \Gamma\Aroid \to \Gamma \Aroid$ given by 
\[
[X, Y]_{\Aroid} \coloneq [\tilde{X},\tilde{Y}]_{T\Groid}\circ \id,
\]
where $\tilde{X}$, $\tilde{Y}$ are the extensions of $X, Y\in \Gamma \Aroid$ to right-invariant vector fields as in equation~\eqref{eq:right.invariant.exstension}. 
\end{theorem}
The following integrable Lie algebroids correspond to the groupoids in Section~\ref{sec:Lie.groupoids}.

\begin{example}\label{ex:aroids}
Let $\Base$ be a smooth manifold and $G$ a Lie group with Lie algebra $\mf g$.
\begin{enumerate}
\item The Lie algebra $\mf g$, with anchor $\anchor:\mf g\to\set{0}$, is a Lie algebroid over a point.  It is the Lie algebroid corresponding to a Lie group over a point.

\item\label{ex:tangent.bundle} The tangent bundle $T\Base$ is a Lie algebroid over $\Base$, with anchor equal to the identity and bracket given by the usual bracket of vector fields.  It is isomorphic to the Lie algebroid of the pair groupoid $\Base\times\Base\rightrightarrows\Base$.

\item The vector bundle $T\Base\oplus(\mf g\times\Base)$ is a transitive Lie algebroid with anchor
\[
\anchor:T\Base\oplus(\mf g\times\Base)\longrightarrow T\Base,
\qquad
S\oplus X\longmapsto S,
\]
and bracket
\[
[S_1\oplus X_1,S_2\oplus X_2]_{\Aroid}
=[S_1,S_2]_{T\Base}\oplus
\big(S_1\cdot X_2-S_2\cdot X_1+[X_1,X_2]_{\mf g}\big),
\]
for $S_1,S_2\in\Gamma T\Base$ and sections $X_1,X_2:\Base\to\mf g\times\Base$.  This is the \emph{trivial Lie algebroid over $\Base$}; its vertical bundle is $\mf g\times\Base$, and it is the Lie algebroid of the trivial Lie groupoid $\Base\times G\times\Base\rightrightarrows\Base$.
Its short exact sequence is
\[
\begin{tikzcd}
\mf g\times\Base \ar[r, "\iota", hook] &
T\Base\oplus(\mf g\times\Base) \ar[r, "\pi_1", two heads] &
T\Base.
\end{tikzcd}
\]

\item\label{ex:action.algebroids} Let $\rho_*:\mf g\to\Gamma T\Base$ be a Lie algebra action on $\Base$.  The trivial vector bundle $\mf g\times\Base\to\Base$, denoted by $\mf g\ltimes\Base$, is the \emph{action Lie algebroid}, with anchor
\[
\anchor(X,x)=(\rho_*X)(x),
\qquad X\in\mf g,\ x\in\Base,
\]
and bracket
\[
[X,Y]_{\Aroid}=[X,Y]_{\mf g}+(\rho_*X)\cdot Y-(\rho_*Y)\cdot X,
\]
for sections $X,Y:\Base\to\mf g\ltimes\Base$.  If the action is transitive, then $\mf g\ltimes\Base$ is transitive and its vertical bundle is the bundle of stabilizer subalgebras.  It is the Lie algebroid of the action Lie groupoid $G\ltimes\Base\rightrightarrows\Base$ associated to a Lie group action $\rho:G\times\Base\to\Base$.
In the transitive case its short exact sequence is
\[
\begin{tikzcd}
\vertbundle \ar[r, "\iota", hook] &
\mf g\ltimes\Base \ar[r, "\anchor", two heads] &
T\Base,
\end{tikzcd}
\]
where $\vertbundle$ is the bundle of stabilizer subalgebras.
\end{enumerate}
\end{example}

A Lie algebroid may be viewed, heuristically, as the infinitesimal $\src$-tangent bundle of a Lie groupoid $\Groid$ after quotienting out its internal symmetries.  From this viewpoint, $\vertbundle$ and $T\Base$ are the two extremal cases: $\vertbundle$ consists of the infinitesimal isotropy directions, whereas $T\Base$ records the base directions obtained after quotienting them out.  A transitive Lie algebroid $\Aroid$ interpolates between these two extremes through the short exact sequence
\[
\begin{tikzcd}
\vertbundle \ar[r, "\iota", hook] &
\Aroid \ar[r, "\anchor", two heads] &
T\Base.
\end{tikzcd}
\]

\begin{example}[The Atiyah algebroid] \label{ex:atiyah.algebroid}
Let $P(\Base,H,\pi)$ be a principal $H$-bundle, with Lie algebra $\mf h$.  The lifted $H$-action on $TP$ defines the quotient vector bundle ${\textstyle\frac{TP}{H}}\to\Base$.  Its sections are identified with the $H$-invariant vector fields on $P$; the bracket is induced by the usual bracket of vector fields, and the anchor
\[
\anchor:{\textstyle\frac{TP}{H}}\longrightarrow T\Base
\]
is induced by $d\pi$.  This is the \emph{Atiyah algebroid} of $P$.

For $V\in\mf h$, the corresponding fundamental vector field is
\[
V^\#(p)=\frac{d}{dt}\Big|_0p\cdot\exp(tV),
\qquad p\in P,
\]
and satisfies
\[
dR_h\big(V^\#(p)\big)
=\big(\Ad_{h^{-1}}V\big)^\#(ph),
\qquad p\in P,\ V\in\mf h,\ h\in H.
\]
Thus the adjoint bundle ${\textstyle\frac{P\times\mf h}{H}}$ is identified with the vertical bundle $\vertbundle$, and the Atiyah algebroid fits into the short exact sequence
\[
\begin{tikzcd}
{\textstyle\frac{P\times\mf h}{H}} \ar[r, "\iota", hook] &
{\textstyle\frac{TP}{H}} \ar[r, "d\pi", two heads] &
T\Base.
\end{tikzcd}
\]
A right splitting $\gamma:T\Base\to{\textstyle\frac{TP}{H}}$ is a horizontal lift and determines the noncanonical vector-bundle decomposition
\[
{\textstyle\frac{TP}{H}}
\simeq
\gamma(T\Base)\oplus
\iota\left({\textstyle\frac{P\times\mf h}{H}}\right).
\]
Equivalently, there is a unique left splitting
\[
\omega:{\textstyle\frac{TP}{H}}
\longrightarrow{\textstyle\frac{P\times\mf h}{H}}
\]
characterized by
\[
\iota\circ\omega+\gamma\circ\anchor=\Id.
\]
In the terminology of \cite{mackenzie_general_2005}, $\omega$ is a \emph{connection reform} in ${\textstyle\frac{TP}{H}}$; it corresponds, upon passage to the quotient, to an $H$-equivariant connection form $\widetilde\omega:TP\to P\times\mf h$.  For further details, see \cite[Sections~3.1, 3.2, 5.2 and 5.3]{mackenzie_general_2005}.  The Lie algebroid of the gauge groupoid in Example~\ref{ex:gauge.groupoid} is the Atiyah algebroid.
\end{example}

The following example is important for introducing connections on Lie algebroids. In particular, the vector bundle of derivations on a vector bundle, has itself the structure of a Lie algebroid. 

\begin{example}[Lie algebroid of derivations on a vector bundle]\label{ex:algebroid.of.derivations}
Let $E\to\Base$ be a vector bundle.  A derivation on $E$ is a pair $(S,D)$, where $S\in\Gamma T\Base$ and $D:\Gamma E\to\Gamma E$ is $\mbb R$-linear and satisfies
\[
D(fX)=(S\cdot f)X+fD(X),
\qquad f\in C^\infty(\Base),\ X\in\Gamma E.
\]
The commutator of derivations is
\[
[(S_1,D_1),(S_2,D_2)]
=\big([S_1,S_2]_{T\Base},D_1D_2-D_2D_1\big).
\]
Let $\Der(E)\to\Base$ denote the vector bundle whose sections are the derivations of $E$.  Locally,
\[
\Der(E)|_U\simeq T\Base|_U\oplus\End(E)|_U.
\]
It is a transitive Lie algebroid, with bracket given by the commutator, anchor $(S,D)\mapsto S$, and short exact sequence
\[
\begin{tikzcd}
\End(E) \ar[r, "\iota", hook] &
\Der(E) \ar[r, "\anchor", two heads] &
T\Base.
\end{tikzcd}
\]
A linear connection $\nabla$ on $E$ is equivalently a right splitting
\[
\nabla:T\Base\longrightarrow\Der(E),
\qquad S\longmapsto(S,\nabla_S),
\]
and hence yields the noncanonical vector-bundle decomposition
\[
\Der(E)\simeq\nabla(T\Base)\oplus\iota\big(\End(E)\big).
\]
Thus $\Der(E)$ plays for $E$ the role played by $\End(V)$ for a vector space $V$; in particular, one may form $\Der(\Aroid)$ for any Lie algebroid $\Aroid$.  The Lie algebroid of the general linear groupoid $\GL(E)$ from Example~\ref{ex.GL(E)} is isomorphic to $\Der(E)$; see \cite[Theorem~3.6.6]{mackenzie_general_2005}.
\end{example}

\subsection{Connections and curvature of anchor-preserving maps}
We now turn to curvature and connections on Lie algebroids.  Following Definition~\ref{def:lie.algebroid.morphism}, we consider bundle maps between Lie algebroids over $\Base$ that preserve the anchor, but not necessarily the bracket.  Their failure to preserve the bracket defines an intrinsic curvature that includes the classical curvature of linear and principal connections.
\begin{definition}\label{def:curvature}
Let $\Aroid_1$ and $\Aroid_2$ be Lie algebroids over $\Base$.  A vector bundle morphism $\varphi:\Aroid_1\to\Aroid_2$ is \emph{anchor-preserving} if
\[
\anchor_2\circ\varphi=\anchor_1,
\qquad
\begin{tikzcd}
\Aroid_1 \ar[r, "\varphi"] \ar[rd, "\anchor_1", swap] &
\Aroid_2 \ar[d, "\anchor_2"] \\
&T\Base.
\end{tikzcd}
\]
The \emph{curvature} of an anchor-preserving map $\varphi$ is the skew-symmetric, $C^\infty(\Base)$-bilinear map
\[
\curv^\varphi:\Aroid_1\oplus\Aroid_1
\longrightarrow\iota_2(\vertbundle_2)\subset\Aroid_2
\]
given by
\[
\curv^\varphi(X,Y)
=\varphi\circ[X,Y]_{\Aroid_1}
-[\varphi(X),\varphi(Y)]_{\Aroid_2},
\qquad X,Y\in\Gamma\Aroid_1.
\]
\end{definition}
\begin{remark}
The Leibniz identity and the anchor-preserving condition imply that the displayed expression is $C^\infty(\Base)$-bilinear.  Moreover,
\begin{align*}
\anchor_2\circ\curv^\varphi(X,Y)
&=\anchor_1\circ[X,Y]_{\Aroid_1}
-[\anchor_2\circ\varphi(X),\anchor_2\circ\varphi(Y)]_{T\Base}\\
&=[\anchor_1X,\anchor_1Y]_{T\Base}
-[\anchor_1X,\anchor_1Y]_{T\Base}=0,
\end{align*}
for $X,Y\in\Gamma\Aroid_1$.  Hence $\curv^\varphi$ defines a bundle map with values in $\iota_2(\vertbundle_2)$.
\end{remark}

Thus an anchor-preserving map $\varphi:\Aroid_1\to\Aroid_2$ is a morphism of Lie algebroids if and only if $\curv^\varphi\equiv0$.

Curvature behaves naturally under composition of anchor-preserving maps.
\begin{lemma}\label{lem:comp.curv}
Let $\Aroid_1$, $\Aroid_2$, and $\Aroid_3$ be Lie algebroids over $\Base$, and let
\[
\varphi_1:\Aroid_1\longrightarrow\Aroid_2,
\qquad
\varphi_2:\Aroid_2\longrightarrow\Aroid_3
\]
be anchor-preserving.  For $\varphi=\varphi_2\circ\varphi_1$, one has
\[
\curv^\varphi
=\varphi_1^*\curv^{\varphi_2}
+\varphi_2\circ\curv^{\varphi_1}.
\]
\end{lemma}
\begin{proof}
\begin{multline*}
\curv^{\varphi}(X, Y) = \varphi_2\circ \varphi_1\circ[X, Y]_{1} - [\varphi_2\circ \varphi_1 (X), \varphi_2\circ \varphi_1(Y)]_3 \\
= \varphi_2\circ \varphi_1\circ[X, Y]_{1} -\varphi_2\circ [\varphi_1(X), \varphi_1(Y)]_2 + \varphi_2\circ [\varphi_1(X), \varphi_1(Y)]_2    - [\varphi_2\circ \varphi_1 (X), \varphi_2\circ \varphi_1(Y)]_3 \\
= \varphi_2 \circ \curv^{\varphi_1}(X,Y) +  (\varphi_1^*\curv^{ \varphi_2})(X,Y).
\end{multline*}
\end{proof}

\begin{example}[Splitting of a Lie algebroid]\label{ex:splitting}
Consider a transitive Lie algebroid $\Aroid$, and a right splitting $\gamma:T\Base \to \Aroid$ of its short exact sequence
\[
\begin{tikzcd}
\vertbundle \ar[r, "\iota", hook] &
\Aroid \ar[r, "\anchor", two heads] &
T\Base \ar[l, "\gamma", bend left = 30, pos = 0.45].
\end{tikzcd}
\]
Then $\gamma$ is anchor-preserving, since $\anchor\circ\gamma=\Id_{T\Base}=\anchor_{T\Base}$.  Such a splitting is a horizontal lift in $\Aroid$ and yields the vector-bundle decomposition
\[
\Aroid\simeq\gamma(T\Base)\oplus\iota(\vertbundle).
\]
Its curvature measures the obstruction to $\gamma$ being a morphism of Lie algebroids and is given by
\[
\curv^\gamma(S_1,S_2)
=\gamma\circ[S_1,S_2]_{T\Base}
-[\gamma(S_1),\gamma(S_2)]_{\Aroid},
\qquad S_1,S_2\in\Gamma(T\Base).
\]
Given a right splitting $\gamma:T\Base \to \Aroid$, there exists a unique left splitting $\omega:\Aroid \to \vertbundle$ such that 
\[
\iota \circ \omega + \gamma \circ \anchor = \Id_{\Aroid},
\]
i.e. we have the diagram
\[
\begin{tikzcd}
\vertbundle \ar[r, "\iota", hook] &
\Aroid \ar[r, "\anchor", two heads] \ar[l, "\omega", bend left=35, pos = 0.55] &
T\Base \ar[l, "\gamma", bend left = 35, pos = 0.45].
\end{tikzcd}
\]
Using the terminology of \cite{mackenzie_general_2005}, $\omega$ is referred to as a \emph{connection reform in $\Aroid$}, see also Example~\ref{ex:atiyah.algebroid}. It has the interpretation of a ``vertical projection'', as it maps $X\in \Aroid$ into the vertical bundle $\vertbundle$. The curvature of $\gamma$ can be written by using $\omega$ as 
\begin{equation}\label{eq:gamma.curvature.form.form}
\curv^\gamma(S_1, S_2) = -\iota\circ \omega[\gamma(S_1), \gamma(S_2)]_{\Aroid},
\end{equation}
which is minus the vertical component of the bracket between horizontal elements. 
Note that $\omega$ is not an anchor-preserving map, as $\anchor_{(\vertbundle)} = \anchor \circ \iota = 0$, so $ \anchor_{(\vertbundle)}\circ \omega \neq \anchor$. In this sense, a right splitting $\gamma:T\Base \to \Aroid$ is more natural to work with in the context of Lie algebroids than a left splitting $\omega:\Aroid \to \vertbundle$.  

\end{example}
\begin{example}\label{rem:curvature.form}
Let $\Aroid={\textstyle\frac{TP}{H}}$ be the Atiyah algebroid of a principal bundle $P(\Base,H,\pi)$, as in Example~\ref{ex:atiyah.algebroid}.  The connection reform
\[
\omega:{\textstyle\frac{TP}{H}}\longrightarrow
{\textstyle\frac{P\times\mf h}{H}}
\]
corresponds to an $H$-equivariant connection form $\widetilde\omega:TP\to P\times\mf h$.  Under this correspondence, the curvature $\curv^\gamma$ in equation~\eqref{eq:gamma.curvature.form.form} is induced by the classical curvature form
\[
\Omega(X_1,X_2)
=-\widetilde\omega[X_1^{\mathrm{hor}},X_2^{\mathrm{hor}}]_{TP}.
\]
Indeed, $\Omega$ is horizontal and $H$-equivariant, and therefore descends to $\curv^\gamma$; see \cite[Section~5.3]{mackenzie_general_2005}.
\end{example}

We now define connections in the setting of Lie algebroids, deviating slightly from the definition of linear connections on a vector bundle. 
Recall from Example~\ref{ex:algebroid.of.derivations} the Lie algebroid $\Der(E)$ of derivations on a vector bundle $E$ over $\Base $. The following definition is taken from \cite[Lecture 2]{fernandes_lectures_2011}.
\begin{definition}\label{def:linear.A.con}
Let $\Aroid$ be a Lie algebroid over $\Base$, and let $E\to\Base$ be a smooth vector bundle. An \emph{$\Aroid$-connection on $E$} is an anchor-preserving map
$ \nabla: \Aroid \to  \Der(E).$ \\
Equivalently, an \emph{$\Aroid$-connection on $E$} is a map 
\[ 
\nabla :  \Aroid \times \Gamma E \to \Gamma E,  
\] 
with the properties that for all $f,g\in C^\infty(\Base)$; $X,Y\in \Aroid$; $s, t \in \Gamma E$,
\begin{enumerate}
\item $\nabla _{(fX + gY)} s = f\nabla_X s + g\nabla_Y s,$
\item $\nabla _X (fs) =  (\anchor X\cdot f) s + f\nabla_X s ,$
\item $\nabla_X(s + t) = \nabla_X s + \nabla_X t.$
\end{enumerate}
\end{definition}
An $\Aroid$-connection $\nabla:\Aroid \to \Der(E)$ on a vector bundle $E$, is called a \emph{representation of $\Aroid$ on $E$} if it is a Lie algebroid morphism, i.e. if it additionally is a homomorphism of the brackets, so that $\curv^{\nabla} =0$. 

For the tangent bundle $T\Base$ of a manifold $\Base$, the above definitions of a connection and of curvature coincide with the usual definitions.
\begin{example}\label{ex:std.curvature}
For the Lie algebroid $T\Base$, with identity anchor and the usual bracket of vector fields, a $T\Base$-connection on $T\Base$ is precisely a linear connection. Its curvature in the sense of Definition~\ref{def:curvature} is
\[
\curv^\nabla(X,Y)
=\nabla_{[X,Y]_{T\Base}}-\nabla_X\nabla_Y+\nabla_Y\nabla_X,
\qquad X,Y\in\Gamma T\Base,
\]
which is the usual curvature, with our sign convention.
\end{example}

\begin{remark}\label{rem:TM.connection}
In \cite{mackenzie_general_2005}, Mackenzie defines a \emph{Lie algebroid connection} on a transitive Lie algebroid
\[
\begin{tikzcd}
\vertbundle \ar[r, "\iota", hook] &
\Aroid \ar[r, "\anchor", two heads] &
T\Base
\end{tikzcd}
\]
to be a right splitting $\gamma:T\Base \to \Aroid$ of the anchor. In particular, for the transitive Lie algebroid of derivations on $\Aroid$
\[
\begin{tikzcd}
\End(\Aroid) \ar[r, hook] &
\Der(\Aroid) \ar[r, two heads, "\anchor_{\Der}"] &
T\Base \ar[l, "\bar{\nabla}", bend left=20, pos=0.55],
\end{tikzcd}
\]
a right splitting $\bar{\nabla}:T\Base \to \Der(\Aroid)$ is, in Mackenzie's terminology, a ``Lie algebroid connection in $\Der(\Aroid)$''. In the terminology of \cite{fernandes_lectures_2011} and Definition~\ref{def:linear.A.con}, however, such a splitting is instead called a $T\Base$-connection on $\Aroid$. 
\end{remark}

\section{Riemannian Lie algebroids and sectional curvature}\label{sec:riemannian.struc.aroids}
A Lie algebroid is, in particular, a vector bundle over a manifold. Equipping it with a smoothly varying inner product---also called a \emph{bundle metric}---provides norms, angles, and horizontal lifts, which are fundamental for studying geodesics, variational principles, and physical models on manifolds. Unlike the tangent bundle, however, an arbitrary vector bundle has no natural bracket of sections and hence no notion of torsion; consequently, a bundle metric alone does not determine a canonical metric connection. The Lie algebroid structure supplies the bracket needed to construct a distinguished connection.

\subsection{Metrics and Levi-Civita connections}
For the rest of this section, let $\Aroid$ be a Lie algebroid over $\Base$, equipped with a smoothly varying inner product $\ip{\cdot}{\cdot}(x)$. This will be called a \emph{Riemannian structure on $\Aroid$} or \emph{a metric on $\Aroid$}. If $\Aroid$ is a Riemannian transitive Lie algebroid there is a canonical right splitting $\gamma$ of the short exact sequence 
\begin{equation}\label{eq:can.splitting}
\begin{tikzcd}
\vertbundle \ar[r, "\iota", hook] & \Aroid \ar[r, "\anchor", two heads]  &T\Base, \ar[l, "\gamma",bend left =25, pos = 0.45]
\end{tikzcd}
\end{equation}
and hence induces the orthogonal vector-bundle decomposition
$$\Aroid \simeq  (\vertbundle)^{\perp} \oplus_{\perp} \vertbundle \simeq \gamma(T\Base) \oplus_{\perp} \iota(\vertbundle) .$$
The vector $\gamma(S) = X$ is the unique $X\in \Aroid$, such that $X \perp \vertbundle$ and $\anchor X= S$.
\begin{definition}\label{def:hor.vert}
In the presence of a right splitting $\gamma$, its image $\gamma(T\Base) \subset \Aroid$ will be called the \emph{horizontal bundle}.
\end{definition}

We define compatibility between an $\Aroid$-connection $\nabla:\Aroid \to \Der(\Aroid)$ and a Riemannian structure on $\Aroid$.x
\begin{definition}\label{def:metric.compatibility}
A derivation $D$ on $\Aroid$ is \emph{metric} if for all $Y,Z \in \Gamma\Aroid$
\[  (\anchor_{\Der} D )\ip{Y}{Z}  = \ip{D\, Y}{Z} + \ip{Y}{D\, Z}.   \]
An $\Aroid$-connection $\nabla$ on $\Aroid$ is \emph{metric} if for all ${X,Y, Z \in \Gamma \Aroid }$ 
\[ \anchor X \ip{Y}{Z}  = \ip{\nabla_X Y}{Z} + \ip{Y}{\nabla_X Z}. \]
\end{definition}
\begin{lemma}[Corollary 3.6.11 in \cite{mackenzie_general_2005}]\label{lem:metric.derivations}
The collection of metric derivations is a subalgebroid of $\Der(\Aroid)$ called the \emph{Lie algebroid of metric derivations on $\Aroid$} and denoted by $\Der_{\mf{so}}(\Aroid)$. Moreover, the vertical bundle in  $\Der_{\mf{so}}(\Aroid)$---which will be denoted by $\mf{so}(\Aroid)$---consists of endomorphisms satisfying
\[  \ip{D\, Y}{Z} + \ip{Y}{D\, Z} = 0, \]
i.e. skew-symmetric endomorphisms with respect to the inner product $\ip{\cdot}{\cdot}$. Hence, the short exact sequence of $\Der_{\so}(\Aroid)$ is
\begin{equation*}
\begin{tikzcd}
\mf{so}(\Aroid) \ar[r, "\iota_{\Der}", hook] & \Der_{\mf{so}}(\Aroid) \ar[r, "\anchor_{\Der}", two heads]  &T\Base.
\end{tikzcd}
\end{equation*}
\end{lemma}
Let $\set{X_j}\subset\Gamma_U\Aroid$ be a local orthonormal frame over a trivializing neighbourhood $U\subset\Base$. Using the metric identification $\Aroid\simeq\Aroid^*$, set
\begin{equation}\label{eq:metric.wedge}
X^k = \ip{\cdot}{X_k},\qquad X_k \wedge X^l = X_k \otimes X^l - X_l \otimes X^k
\end{equation}

so that $X_k\wedge X^l=-X_l\wedge X^k$. These endomorphisms span $\mf{so}(\Aroid)$ locally, and hence
\[
\Der_{\so}(\Aroid)|_U\simeq T\Base|_U\oplus\so(\Aroid)|_U,
\qquad
D=S\oplus\ChEnd
=S\oplus\sum_{j<k}\ChEnd_{\,j}^{\,k}X_k\wedge X^j.
\]

\begin{definition}\label{def:torsion}
The \emph{torsion} of an $\Aroid$-connection $\nabla$ on $\Aroid$ is the skew-symmetric, $C^\infty(\Base)$-bilinear map
\[T^\nabla:  \Gamma\Aroid \times \Gamma\Aroid \to \Gamma\Aroid, \quad (X,Y)  \mapsto  \nabla_X Y - \nabla_Y X - [X, Y]_{\Aroid}.  \]
Since the map is $C^\infty$-bilinear, it defines $T^\nabla:\Aroid \oplus \Aroid \to \Aroid$. 
An $\Aroid$-connection on $\Aroid$ is \emph{torsion-free} if $T^\nabla \equiv 0$.
\end{definition}

\begin{theorem}[Lecture 2 in \cite{fernandes_lectures_2011}]\label{th:levi.civita.con}
Let $\Aroid$ be a finite-dimensional Lie algebroid with a Riemannian structure $\ip{\cdot}{\cdot}$. Then there exists a unique $\Aroid$-connection $\nabla$ on $\Aroid$, called the \emph{Levi-Civita connection on $\Aroid$}, which is metric and torsion-free,
\begin{equation*}
\begin{tikzcd}
\Der_{\mf{so}}(\Aroid) \ar[rd, "\anchor_{\Der}", two heads] & \\
\Aroid \ar[r, "\anchor"] \ar[u, "\nabla", dashed]  &T\Base.
\end{tikzcd}
\end{equation*}
The connection is given by the usual Koszul formula
\begin{equation*}
2\ip{ \nabla_X Y}{Z } =
\anchor X\ip{ Y}{Z} + \anchor Y\ip{ Z}{X} - \anchor Z\ip{ X}{Y}
+ \ip{ [X,Y]}{Z} - \ip{ [Y,Z]}{X} + \ip{ [Z,X]}{Y}.
\end{equation*}

\end{theorem}
\begin{proof}
The Koszul formula follows from the usual algebraic manipulations involving the metric compatibility condition and the torsion-free condition. Existence is proved in local coordinates as in the case when the Lie algebroid is the tangent bundle.
\end{proof}

For an infinite-dimensional Lie algebroid, the same conclusion holds provided that the right-hand side of the Koszul formula is represented, for every $X,Y\in\Gamma\Aroid$, by a smooth section $\nabla_XY$ depending smoothly on $X$ and $Y$. Nondegeneracy of the metric implies uniqueness of such a section, but for a weak metric it does not imply its existence. Thus, in the infinite-dimensional setting, existence of the Levi-Civita connection must be verified in each application.

\begin{remark}\label{rem:metric.derivations.splitting}
For a transitive Lie algebroid with a Riemannian structure
\[
\begin{tikzcd}
\vertbundle \ar[r, "\iota", hook] & \Aroid \ar[r, "\anchor", two heads]  &T\Base, \ar[l, "\gamma",bend left =25, pos = 0.45]
\end{tikzcd}
\]
the composition $\begin{tikzcd}
\nabla\circ \gamma:T\Base \ar[r, "\gamma"] & 
\Aroid \ar[r, "\nabla"] &
\Der_{\so}(\Aroid)
\end{tikzcd}$ 
induces a right splitting of the short exact sequence 
\begin{equation}\label{eq:deriv.splitting}
\begin{tikzcd}
\mf{so}(\Aroid) \ar[r, "\iota_{\Der}", hook] & \Der_{\mf{so}}(\Aroid) \ar[r, "\anchor_{\Der}", two heads]  &T\Base, \ar[l, "\nabla \circ\gamma",bend left , pos = 0.55, dashed] 
\end{tikzcd}
\end{equation}
and $\nabla\circ\gamma$ is called a metric $T\Base$-connection on $\Aroid$, see Remark \ref{rem:TM.connection}. This is a horizontal lift from $T\Base$ to $\Der_{\so}(\Aroid)$, as it gives a horizontal complement to the vertical bundle $\so(\Aroid) = \Ker \anchor_{\Der}$, and is what is commonly known as a ``linear connection'' on the vector bundle $\Aroid$ over $\Base$, see e.g. \cite{kobayashi_foundations_1963}. Such a splitting does not give information on how to associate a vector in the vertical bundle $\vertbundle$ to a derivation on $\Aroid$, in contrast to the Levi-Civita connection $\nabla:\Aroid \to \Der_{\mf{so}}(\Aroid)$. 
\end{remark}

\begin{remark}\label{rem:curv.levi-civita}
By Definition~\ref{def:curvature} the curvature of the Levi-Civita connection on $\Aroid$
is for $X,Y \in \Gamma \Aroid$ given by
\[
\curv^{\nabla}(X,Y) = \nabla_{[X,Y]_{\Aroid}} - \nabla_X\nabla_Y + \nabla_Y\nabla_X.
\]
\end{remark}

\begin{example}
Let $(\Base, g)$ be a Riemannian manifold. By example~\ref{ex:aroids}.\ref{ex:tangent.bundle}., $T\Base$ is a Lie algebroid. The Levi-Civita connection on $T\Base$ from Theorem~\ref{th:levi.civita.con} is the standard Levi-Civita connection on $(\Base, g)$. The Riemannian curvature of $\Base$ in the classical sense, coincides with the curvature of the Levi-Civita connection on $T\Base$,
\[\curv^{\Base} = \curv^\nabla : T\Base \oplus T\Base \to \iota_{\Der}(\so(T\Base)).  \]
\end{example}

\subsection{Source-fibre metrics and isotropy}\label{sec:relation.between.riemannian.structures}\label{sec:Levi.civita.on.vertical}
Let $\Aroid$ be the Lie algebroid of $\Groid\rightrightarrows\Base$.  A metric on $\Aroid$ extends uniquely, by right translation, to a metric on $T^{\src}\Groid$: for $g\in\Groid$ with $\trg(g)=y$ and $\widetilde X,\widetilde Y\in T_g^{\src}\Groid$, set
\[
\ip{\widetilde X}{\widetilde Y}_{\Groid}(g)
=\ip{dR_{g^{-1}}\widetilde X}{dR_{g^{-1}}\widetilde Y}_{\Aroid}(y).
\]
Conversely, a right-invariant metric on $T^{\src}\Groid$ restricts along the unit section to a metric on $\Aroid$, and these constructions are inverse.  Since the metric and bracket of right-invariant source-tangent fields are determined at the units, their Levi-Civita connection and curvature are likewise determined by the Levi-Civita connection and curvature of $\Aroid$.

\begin{example}[Riemannian structure on principal bundles]\label{rem:atiyah.correspondence}
Suppose that $(\src,\trg):\Groid\to\Base\times\Base$ is a surjective submersion and fix $x\in\Base$.  Then $P=\Groid_x$ is a principal $H$-bundle over $\Base$, where $H=\Groid_x^x$, with Lie algebra $\mf h$, and the bundle projection is $\pi=\trg|_P$.  As recalled in Examples~\ref{ex:gauge.groupoid} and~\ref{ex:atiyah.algebroid}, the Lie algebroid $\Aroid$ of $\Groid$ is isomorphic to the transitive Atiyah algebroid in the short exact sequence
\[
\begin{tikzcd}
{\textstyle\frac{P\times\mf h}{H}} \ar[r,"\iota",hook] &
{\textstyle\frac{TP}{H}} \ar[r,"d\pi",two heads] &
T\Base .
\end{tikzcd}
\]
Under this identification, the preceding correspondence identifies a Riemannian metric on $\Aroid\simeq{\textstyle\frac{TP}{H}}$ with an $H$-invariant Riemannian metric on the principal bundle $P$.
\end{example}

Let $\Groid\rightrightarrows\Base$ be a transitive Lie groupoid with a right-invariant metric, and let $\Aroid$ be its Riemannian Lie algebroid.  The isotropy groups $H_x=\src^{-1}(x)\cap\trg^{-1}(x)$ form a smooth bundle of Lie groups, whose Lie algebra bundle is the vertical bundle $\vertbundle$.  The bracket and metric on $\vertbundle$ vary smoothly with $x$; hence their fibrewise Koszul formulas define a smooth family of Levi-Civita connections on the groups $H_x$.  The following proposition identifies this family with the Levi-Civita connection of the Riemannian Lie algebroid $\vertbundle$.
\begin{proposition}\label{prop:adbundle.levi.civita}
For every $x\in\Base$, the Levi-Civita connection $\widehat{\nabla}$ on $\vertbundle$, restricted to $\vertbundle_x=\mf h_x$, agrees with the Levi-Civita connection of $H_x$ at the identity, evaluated on right-invariant vector fields.  Under the identification $T_{\Id}H_x\simeq\mf h_x$, one likewise has
\[
\curv^{\widehat{\nabla}}|_x=\curv^{H_x}|_{\Id}.
\]
\end{proposition}
\begin{proof}
The anchor of $\vertbundle$ vanishes, so its bracket and Levi-Civita connection are $C^\infty(\Base)$-bilinear.  Consequently, the Koszul formula on $\vertbundle$, evaluated at $x$, is precisely the Koszul formula for the right-invariant metric on $H_x$ at the identity.  This proves the assertion for the connections; substitution into the definition of curvature gives the curvature identity.
\end{proof}

An example of the use of Proposition~\ref{prop:adbundle.levi.civita} is given in Section~\ref{ex:RP1}.

\subsection{The metric-adjoint of the adjoint map}
We use the algebroid bracket to give an explicit formula for the Levi-Civita connection on a Lie groupoid with right-invariant metric. First we introduce the \emph{metric-adjoint} of the adjoint map $\ad:\Gamma\Aroid \to \Gamma\Der(\Aroid)$, $X\mapsto [X, \cdot]$.
\begin{definition}\label{def:metric.adjoint}
Let $(\Aroid,\ip{\cdot}{\cdot})$ be a Riemannian Lie algebroid. Suppose that, for every $X,Z\in\Gamma\Aroid$, there exists a smooth section $\ad^\madj_XZ\in\Gamma\Aroid$ satisfying
\[
\ip{Y}{\ad^\madj_XZ}
=\anchor X\cdot\ip{Y}{Z}-\ip{\ad_XY}{Z},
\qquad Y\in\Gamma\Aroid,
\]
and that these sections define a smooth map $\ad^\madj:\Gamma\Aroid\to\Gamma\Der(\Aroid)$. This map is called the \emph{metric-adjoint of the adjoint map}. Equivalently,
\[
\anchor X\cdot\ip{Y}{Z}
=\ip{\ad_XY}{Z}+\ip{Y}{\ad^\madj_XZ}.
\]
For finite-dimensional Lie algebroids its existence is automatic. For infinite-dimensional Lie algebroids with a weak metric, it is an additional condition that must be verified in each application.
\end{definition}

\begin{remark}
Let $G$ be a Lie group with a right-invariant metric $\ip{\cdot}{\cdot}$, and Lie algebra~$\mf g$. The above definition is a Lie algebroid version (up to sign) of the operator $B:\mf g \times \mf g \to \mf g$
\[
\ip{B(c,a)}{b} = \ip{[a,b]}{c}, \quad \text{for }a,b,c \in \mf g,
\]
 popularized by Arnold in \cite{arnold_sur_1966} and \cite{arnold_mathematical_1989}. 
\end{remark}
The metric-adjoint of a map generalizes the (minus) transpose of a linear operator on a vector space into the setting of vector bundles. Moreover, using the musical isomorphism  $\Aroid^*\simeq \Aroid$ induced by the metric $\ip{\cdot}{\cdot}$ on $\Aroid$, the metric-adjoint of the adjoint map, $\ad^\madj$, coincides with the formal Lie algebroid analogue of the ``Lie derivative'', as is made clear by the following proposition.

\begin{proposition}\label{prop:metric.adjoint.properties}
The metric-adjoint of the adjoint map, $\ad^\madj:\Gamma\Aroid \to \Gamma\Der(\Aroid)$ has the following properties
\begin{enumerate}
\item It is a homomorphism of the bracket structures of $\Aroid$ and $\Der(\Aroid)$, i.e.  
\[
\ad^\madj_{[X,Y]} = \ad^\madj_X \ad^\madj_Y - \ad^\madj_Y \ad^\madj_X.
\]
\item $\ad^\madj:\Gamma\Aroid \to \Gamma\Der(\Aroid)$ and $\ad^\madj_X:\Gamma\Aroid \to \Gamma \Aroid$ are both $\mbb R$-linear, and for $f\in C^\infty(\Base)$ and $X,Y,Z\in \Gamma \Aroid$
\[
\begin{gathered}
\ip{\ad^\madj_{X} fY}{Z} = (\anchor X\cdot f)\ip{Y}{Z} + f\ip{\ad^\madj_{X} Y }{Z},\\
\ip{\ad^\madj_{fX} Y}{Z} = (\anchor Z\cdot f)\ip{X}{Y} + f\ip{\ad^\madj_{X} Y }{Z},
\end{gathered}
\]
which shows that $\ad^\madj_X$ is indeed a section of derivations on $\Aroid$.
\item Under the musical isomorphism $\Aroid^* \simeq \Aroid$, the metric-adjoint coincides with the ``Lie derivative'' of a section of $\Aroid^*$. I.e. if $\alpha \in \Gamma \Aroid^*$, such that $\alpha(X)= \ip{X}{\alpha^\sharp} $, then
\[
\ip{\ad^\madj_X\alpha^{\sharp}}{Z}  = (\Lie_X\alpha )(Z).
\]
\item Given an orthonormal frame $\set{X_k}$ for $\Aroid$, such that $\ad_{X_k}$ is given by a matrix $(\ad_{X_k})$, then the corresponding matrix for $\ad^\madj_{X_k}$ is the negative transpose of $(\ad_{X_k})$:
\[
(\ad^\madj_{X_k}) = -(\ad_{X_k})^\trans.
\]
\item\label{lem:metric.adjoint.property.horizontal.ideal} For $X\in \Gamma \Aroid$, the composed map $\ad^\madj_{X}\circ\gamma:\Gamma T\Base \to  \Gamma\Aroid$, takes values in the horizontal distribution, i.e. $\ad^\madj_X\circ  \gamma: \Gamma T\Base \to \Gamma(\gamma(T\Base) )$.
\end{enumerate}
\end{proposition}
\begin{proof}
\begin{enumerate}
\item We use the fact that $\ad$ is a homomorphism, and compute for 
$X,Y,Z,W \in \Gamma\Aroid$
\[ 
\ip{\ad^\madj_{[X,Y]}Z}{W} = \anchor[X,Y]\cdot \ip{Z}{W} - \ip{Z}{\ad_{X}\ad_Y W - \ad_Y\ad_X W}.
  \]
From the other side, 
\begin{align*}
\ip{\ad&^\madj_X \ad^{\madj}_Y Z}{W} - \ip{\ad^\madj_Y\ad^\madj_X Z}{W} \\
= &\anchor X\cdot \ip{\ad^\madj_Y Z}{W} -\ip{\ad^\madj_Y Z}{\ad_X W} - \anchor Y\cdot \ip{\ad^\madj_X Z}{W} + \ip{\ad^\madj_X Z}{\ad_Y W}\\
=&\anchor X \anchor Y\cdot\ip{Z}{W} - \anchor X\cdot \ip{Z}{\ad_Y W} - \anchor Y\cdot \ip{Z}{\ad_X W} + \ip{Z}{\ad_Y \ad_X W}\\
 &-\anchor Y \anchor X\cdot \ip{Z}{W} + \anchor Y\cdot \ip{Z}{\ad_X W} + \anchor X \cdot \ip{Z}{\ad_Y W} - \ip{Z}{\ad_X \ad_Y W}\\
=&(\anchor X \anchor Y-\anchor Y \anchor X)\cdot\ip{Z}{W} - \ip{Z}{\ad_X \ad_Y W -\ad_Y \ad_X W},
\end{align*}
proving (1). 
\item Follows from direct computations using Definition~\ref{def:metric.adjoint}. 
\item  The  ``Lie derivative'' $\Lie_X$  (see \cite[Proposition~7.1.4]{mackenzie_general_2005} and the discussion thereafter) in the setting of Lie algebroids, is characterized by: for $X,Y\in \Gamma\Aroid$, $f\in C^\infty(\Base)$
\[
\Lie_X(f) = \anchor X\cdot f, \quad \text{and }\quad \Lie_X Y = [X,Y],
\]
and on sections of tensors of $\Aroid$ it is extended by the Leibniz rule
\[
\Lie_X(T\otimes S) = \Lie_X T \otimes S + T\otimes \Lie_X S.
\]
Let $\alpha \in \Gamma \Aroid^*$, $X\in \Gamma\Aroid$ and $Z\in \Aroid$. Applying the above axioms of the Lie derivative $\Lie_X$, we get
\[
\Lie_X\alpha (Z) =  \anchor X \cdot \alpha(Z) -\alpha(\ad_X Z)  = \anchor X \cdot \ip{\alpha^\sharp}{Z} - \ip{\alpha^{\sharp}}{\ad_X Z}  = \ip{\ad^\madj_X\alpha^{\sharp}}{Z}.
\]
\item Given an orthonormal frame $\set{X_k}$, consider
\[
\ip{\ad^\madj_{X_k} X_l}{X_j} = \anchor X_k\cdot \ip{X_l}{X_j} - \ip{X_l}{\ad_{X_k}X_j}.
\]
The product $\ip{X_l}{X_j}$ is a constant function in an orthonormal frame, so $\anchor X_k\cdot \ip{X_l}{X_j} = 0$. Writing in matrix notation, the metric is the identity matrix in the orthonormal frame, hence
\[
((\ad^\madj_{X_k}) X_l )^\trans X_j=X_l^\trans(\ad^\madj_{X_k})^\trans X_j =- X_l^\trans (\ad_{X_k})X_j.
\]
\item Let $X\in \Gamma\Aroid$, $S\in \Gamma T\Base$ and $V\in \vertbundle$, then 
\[
\ip{\ad^{\madj}_X \gamma(S) }{\iota V} = \anchor X \cdot \ip{\gamma(S)}{\iota(V)} - \ip{\gamma(S)}{\ad_X \iota(V)}.
\]
The first term vanishes since $\gamma(S)\perp \iota(V)$, and the second term vanishes since the vertical bracket is an ideal, i.e. $\anchor (\ad_X \iota(V) ) = 0$.
\end{enumerate}
\end{proof}

\subsection{The sectional curvature formula}
We give an explicit formula for the Levi-Civita connection on a Lie Groupoid with right-invariant  source-fibre metric, in terms of $\ad$ and $\ad^\madj$. 
\begin{lemma}\label{lem:levi.civita.formula}
Let $\tilde{X}(g)= dR_gX$ and $\tilde{Y}= dR_g Y$ be orthogonal right-invariant vector fields on the Lie groupoid $\Groid$ endowed with a right-invariant source-fibre metric. Then the vector field $\nabla^{\Groid}_{\tilde{X}}\tilde{Y}$ is also right-invariant, and at the identity $\id_{\Base}\subset \Groid$ given by 
\[
\nabla^{\Groid}_{\tilde{X}} \tilde{Y}\big|_{\id_{\Base}}  = \nabla_X Y= \frac{1}{2}\left( \ad_X Y + \ad^\madj_X Y + \ad^\madj_Y X \right).
\]
\end{lemma}
\begin{proof}
Using Theorem~\ref{th:levi.civita.con}, we write for $X,Y,Z\in \Gamma \Aroid$, with $X$ and $Y$ orthogonal
\[
2\ip{\nabla_X Y}{Z} = \ip{ [X,Y]}{Z}  +\anchor X\ip{ Y}{Z}  -\ip{ [X,Z]}{Y} + \anchor Y\ip{ Z}{X} - \ip{ [Y,Z]}{X}
\]
\[
= \ip{\ad_X Y}{Z} + \ip{\ad^\madj_X Y }{Z} + \ip{\ad^\madj_Y X}{Z}.
\]
By the non-degeneracy of the metric, we obtain the desired result.
\end{proof}
The Riemannian curvature of a Lie algebroid can likewise be written solely in terms of its bracket and metric. For Lie groups, this recovers the Arnold-Milnor 1-2-3-4 formula for sectional curvature; see
Theorem~\ref{th:arnold.1-2-3-4}. We now prove the corresponding Lie groupoid
version.
\begin{theorem}[2-1-2-3-4 formula]\label{th:2-1-2-3-4}
The sectional curvature of a Lie groupoid $\Groid$ with a right-invariant metric, in the direction determined by an orthonormal pair $\tilde{X},\tilde{Y}\in T^{\src}_g\Groid$, is
\[
C^{\Groid}(\tilde{X}, \tilde{Y}) =2\big(\anchor X\cdot \ip{\alpha }{ Y} - \anchor Y \cdot \ip{\alpha}{X} \big) 
+ \ip{\delta}{\delta} + 2\ip{\alpha}{\beta} - 3\ip{\alpha}{\alpha} -4\ip{B_X}{B_Y},
\]
where $X,Y\in\Gamma\Aroid$ are local orthonormal sections defined near $y=\trg(g)$ and satisfying
\[
X_y=dR_{g^{-1}}\tilde X,
\qquad
Y_y=dR_{g^{-1}}\tilde Y,
\]
and
\[
2\alpha=[X,Y]_{\Aroid}=\ad_XY,
\qquad
2\delta=\ad^\madj_XY+\ad^\madj_YX,
\]
\[
2\beta=\ad^\madj_XY-\ad^\madj_YX,
\qquad
2B_X=\ad^\madj_XX,
\qquad
2B_Y=\ad^\madj_YY.
\]
The right-hand side is independent of the chosen local orthonormal extensions.
\end{theorem}
The proof of Theorem~\ref{th:2-1-2-3-4} is given in Section~\ref{sec:proofs.of.main.theorems}.

Compared with the Arnold-Milnor formula in
Theorem~\ref{th:arnold.1-2-3-4}, the formula for a Lie groupoid with a
source-fibre metric contains two additional terms differentiating the
coefficients of the algebroid bracket along the anchor images of $X,Y$.
In particular, these terms vanish when the bracket coefficients in the
chosen orthonormal frame are constant.
For a Lie group over a point and the pair groupoid of a Riemannian manifold, the formula specializes as follows.

\begin{corollary}[\cite{arnold_sur_1966}]\label{cor:arnold.2-1-2-3-4}
For a Lie group $G$ with a right-invariant metric and an orthonormal pair $\tilde X,\tilde Y\in T_gG$, Theorem~\ref{th:2-1-2-3-4} reduces to Arnold's formula
\[
C^G(\tilde X,\tilde Y)
=\ip{\delta}{\delta}+2\ip{\alpha}{\beta}
-3\ip{\alpha}{\alpha}-4\ip{B_X}{B_Y},
\]
where $X=dR_{g^{-1}}\tilde X$, $Y=dR_{g^{-1}}\tilde Y$, and $\alpha,\beta,\delta,B_X,B_Y$ are as in Theorem~\ref{th:2-1-2-3-4}, with $[X,Y]_{\Aroid}=[X,Y]_{\mf g}$ and $\ad^\madj=-\ad^\trans$.
\end{corollary}

\begin{corollary}\label{cor:Riemannian.2-1-2-3-4}
For a Riemannian manifold $(\Base,\ip{\cdot}{\cdot})$ and an orthonormal pair $\tilde X,\tilde Y\in T_x\Base$, Theorem~\ref{th:2-1-2-3-4} gives
\[
C^{T\Base}(\tilde X,\tilde Y)
=2\bigl(X\cdot\ip{\alpha}{Y}-Y\cdot\ip{\alpha}{X}\bigr)
+\ip{\delta}{\delta}+2\ip{\alpha}{\beta}
-3\ip{\alpha}{\alpha}-4\ip{B_X}{B_Y},
\]
where $X,Y$ are local orthonormal extensions of $\tilde X,\tilde Y$, and the remaining notation is that of Theorem~\ref{th:2-1-2-3-4}, with $\Aroid=T\Base$, $[X,Y]_{\Aroid}=[X,Y]_{T\Base}$ the bracket of vector fields, and $\anchor=\Id$.
\end{corollary}

\section{Curvature decomposition for Riemannian submersion Lie algebroids}\label{sec:Riemannian.submersion.Lie.algebroids}
\subsection{Orthogonal splitting and metric derivations}
Let
\[
\begin{tikzcd}
\vertbundle \ar[r,"\iota",hook] &
\Aroid \ar[r,"\anchor",two heads] &
T\Base
\end{tikzcd}
\]
be a transitive Riemannian Lie algebroid.  Its metric determines the orthogonal right splitting $\gamma$ and the vector-bundle decomposition
\[
\Aroid\simeq\gamma(T\Base)\oplus{\perp}\iota(\vertbundle),
\qquad
\gamma(T\Base)=\iota(\vertbundle)^\perp.
\]
The induced metrics are
\[
\ip{S_1}{S_2}_{T\Base} \coloneq
\ip{\gamma(S_1)}{\gamma(S_2)}_{\Aroid},
\qquad
\ip{V_1}{V_2}_{\vertbundle} \coloneq
\ip{\iota(V_1)}{\iota(V_2)}_{\Aroid}.
\]
Thus $\anchor|_{\gamma(T\Base)}$ and $\iota$ are isometries.

\begin{definition}\label{def:RSLA}
A transitive Riemannian Lie algebroid equipped with these induced metrics and orthogonal splitting is called a \emph{\RSLA}.
\end{definition}

\begin{remark}
Under the Atiyah correspondence of Remark~\ref{rem:atiyah.correspondence}, a Riemannian submersion Lie algebroid metric corresponds to an $H$-invariant bundle-like metric on the principal bundle $P(\Base,H,\pi)$.
\end{remark}
The orthogonal splitting of $\Aroid$ induces the vector-bundle decomposition
\[
\so(\Aroid)
\simeq \so(T\Base)\oplus\so(\vertbundle)
\oplus\Hom(T\Base,\vertbundle).
\]
The corresponding block decomposition of metric derivations relates the Levi-Civita connections on $\Aroid$, $T\Base$, and $\vertbundle$.

\begin{theorem}\label{th:block.diag.derivation}
For $D\in\Gamma\Der_{\so}(\Aroid)$, write
\[
(D)=
\begin{pmatrix}
\anchor\circ D\circ\gamma & \anchor\circ D\circ\iota\\
\omega\circ D\circ\gamma & \omega\circ D\circ\iota
\end{pmatrix}
\]
with respect to the orthogonal vector-bundle decomposition
$\Aroid\simeq\gamma(T\Base)\oplus_{\perp}\iota(\vertbundle)$.
\begin{enumerate}
\item\label{th:levi.civita.to.horizontal}
The diagonal block
\[
g:\Der_{\so}(\Aroid)\longrightarrow\Der_{\so}(T\Base),
\qquad D\longmapsto\anchor\circ D\circ\gamma,
\]
is anchor-preserving and projects the Levi-Civita connection on $\Aroid$ to that on $T\Base$:
\[
\widetilde{\nabla}=g\circ\nabla\circ\gamma,
\qquad
\widetilde{\nabla}_S=g(\nabla_{\gamma(S)}),
\quad S\in\Gamma T\Base.
\]
\item
Likewise,
\[
\widehat g:\Der_{\so}(\Aroid)\longrightarrow\Der_{\so}(\vertbundle),
\qquad D\longmapsto\omega\circ D\circ\iota,
\]
is anchor-preserving and projects the Levi-Civita connection on $\Aroid$ to that on $\vertbundle$:
\[
\widehat{\nabla}=\widehat g\circ\nabla\circ\iota,
\qquad
\widehat{\nabla}_V=\widehat g(\nabla_{\iota(V)}),
\quad V\in\Gamma\vertbundle.
\]
\item\label{th:off.diag.cinfty.linear}
The off-diagonal entries are $C^\infty(\Base)$-linear and hence define vector-bundle maps
\[
\omega\circ D\circ\gamma:T\Base\longrightarrow\vertbundle,
\qquad
\anchor\circ D\circ\iota:\vertbundle\longrightarrow T\Base.
\]
\item\label{th:off.diag.negative.metric.adjoint}
These maps are negative fibrewise adjoints: for $S\in T\Base$ and $V\in\vertbundle$,
\[
\ip{\omega\circ D\circ\gamma(S)}{V}_{\vertbundle}
=-\ip{S}{\anchor\circ D\circ\iota(V)}_{T\Base}.
\]
\end{enumerate}
\end{theorem}

The proof of the theorem is given in Section~\ref{sec:proofs.of.main.theorems}. 

\begin{remark}\label{rem:second.fundamental.form}
Let $\Groid\rightrightarrows\Base$ integrate the \RSLA{} $\Aroid$.  The inclusion $\Groid_x^x\hookrightarrow\Groid_x$ has tangent inclusion $\iota:\vertbundle_x\hookrightarrow\Aroid_x$ at the identity, and its second fundamental form is
\[
\II(V_1,V_2)
=\anchor\nabla_{\iota(V_1)}\iota(V_2)
=\frac12\anchor\bigl(
\ad^\madj_{\iota(V_1)}\iota(V_2)
+\ad^\madj_{\iota(V_2)}\iota(V_1)
\bigr).
\]
For fixed $V_1$, this is the upper-right off-diagonal block of $\nabla_{\iota(V_1)}$, and post-composing with $\gamma$ lifts it to the O'Neill tensor $T$ evaluated on vertical vectors. 
\end{remark}

\begin{remark}\label{rem:horizontal.to.vertical.second.fundamental}
The lower-left off-diagonal block of $\nabla_{\gamma(S_1)}$ is the $C^\infty(\Base)$-bilinear, skew-symmetric map
\[
\omega\circ \nabla_{\gamma(S_1)}\gamma(S_2)
=\frac12\omega\circ \ad_{\gamma(S_1)}\gamma(S_2)
=-\frac12\omega\circ \curv^\gamma(S_1,S_2).
\]
Here the metric-adjoint terms in Lemma~\ref{lem:levi.civita.formula} are horizontal by Proposition~\ref{prop:metric.adjoint.properties}~(\ref{lem:metric.adjoint.property.horizontal.ideal}).  Under the inclusion $\iota$, this map is the restriction of the O'Neill tensor $A$ to horizontal vectors.
\end{remark}

\subsection{Curvature factorization and the projection maps}
We now state one of the main theorems of the paper: an abstract Lie
algebroid analogue of O'Neill's curvature identities. 
Boucetta states O'Neill's formulas for sectional curvature of Riemannian
Lie algebroids~\cite{boucetta_riemannian_2011}.
The maps entering our curvature identities are summarized by the diagram
\begin{equation*}
\begin{tikzcd}
\Der_{\so}(\vertbundle)
& \Der_{\mf{so}}(\Aroid) \ar[r, "g"] \ar[l, "\widehat{g}", swap]
& \Der_{\so}(T\Base) \\
\vertbundle \ar[r, "\iota", hook] \ar[u, "\widehat{\nabla}"]
& \Aroid \ar[r, "\anchor", two heads] \ar[u, "\nabla"]
  \ar[l, "\omega", bend left=25, pos=0.55]
& T\Base. \ar[l, "\gamma", bend left=25, pos=0.45]
  \ar[u, "\widetilde{\nabla}"]
\end{tikzcd}
\end{equation*}
\begin{theorem}\label{th:algebroid.ONeill}
Let $\Aroid$ be a \RSLA{}, with the induced Levi-Civita connections
$\nabla$, $\widetilde\nabla$, and $\widehat\nabla$, and let $g,\widehat g$
be the maps of Theorem~\ref{th:block.diag.derivation}.  Then the curvature
of the induced Levi-Civita connection on the base is the sum of the
curvature contributions arising from $g$, $\nabla$, and $\gamma$, while
the vertical curvature is the sum of those arising from $\widehat g$ and
$\nabla$.  More precisely,
\begin{equation}\tag{\emph{i}}\label{eq:lie.algebroid.oneill.base}
\curv^{\widetilde{\nabla}} = (\nabla \circ \gamma )^* \curv^g +  g\circ \gamma^* \curv^{\nabla}  + g\circ \nabla \circ \curv^\gamma.  
\end{equation}
\begin{equation}\tag{\emph{ii}}\label{eq:lie.algebroid.oneill.fibre}
\curv^{\widehat{\nabla}} = (\nabla \circ \iota )^* \curv^{\widehat{g}} +  \widehat{g}\circ \iota^* \curv^{\nabla} . 
\end{equation}
\end{theorem}
\begin{proof}
By Theorem~\ref{th:block.diag.derivation}~(\ref{th:levi.civita.to.horizontal}),
$\widetilde\nabla=g\circ\nabla\circ\gamma$; applying
Lemma~\ref{lem:comp.curv} gives \emph{(i)}.  Similarly,
$\widehat\nabla=\widehat g\circ\nabla\circ\iota$, and \emph{(ii)} follows
from the same lemma because $\iota$ is a Lie algebroid morphism and hence
$\curv^\iota=0$.
\end{proof}

Equations~\eqref{eq:lie.algebroid.oneill.base} and
\eqref{eq:lie.algebroid.oneill.fibre} are Lie algebroid analogues of
O'Neill's formulas for a Riemannian submersion; see
Theorem~\ref{th:oneill.formulas.riemannian.submersion}.  They separate
the contributions of the splitting, the Riemannian curvature of $\Aroid$,
and the projections between derivation algebroids.  The curvatures of $g$ and $\widehat g$ are
computed in the following theorem.

\begin{theorem}\label{th:curvature.maps.between.derivations}
Let $D_1,D_2\in\Der_{\so}(\Aroid)$, and write
$\Skew(L)=\frac12(L-L^\trans)$.
\begin{enumerate}
\item\label{th:g.curvature.formula}
The curvature of $g(D)=\anchor\circ D\circ\gamma$ is twice the
rotational part of the induced endomorphism
$(\anchor\circ D_1\circ\iota)\circ(\omega\circ D_2\circ\gamma)$:
\[
\curv^g(D_1,D_2)
=2\Skew\bigl((\anchor\circ D_1\circ\iota)
\circ(\omega\circ D_2\circ\gamma)\bigr).
\]
In particular, $\curv^g\equiv0$ if and only if
$\dim(\Base)\leq1$ or $\rank(\vertbundle)=0$.

\item\label{th:horisontal.derivation.curv.formula}
For $S_1,S_2\in\Gamma T\Base$, the horizontal derivations
$D_i=\nabla_{\gamma(S_i)}$ satisfy
\[
\curv^g(\nabla_{\gamma(S_1)},\nabla_{\gamma(S_2)})
=\frac12\Skew\bigl((\curv^\gamma_{S_2})^\trans
\curv^\gamma_{S_1}\bigr),
\]
where, for $i=1,2$,
\[
\curv^\gamma_{S_i}
=\omega\circ\curv^\gamma(S_i,\,\cdot)
=-\omega\circ\ad_{\gamma(S_i)}\circ\gamma
:T\Base\longrightarrow\vertbundle.
\]

\item
The curvature of $\widehat g(D)=\omega\circ D\circ\iota$ is likewise
twice the rotational part of the induced endomorphism
$(\omega\circ D_1\circ\gamma)\circ(\anchor\circ D_2\circ\iota)$:
\[
\curv^{\widehat g}(D_1,D_2)
=2\Skew\bigl((\omega\circ D_1\circ\gamma)
\circ(\anchor\circ D_2\circ\iota)\bigr).
\]
In particular, $\curv^{\widehat g}\equiv0$ if and only if
$\dim(\Base)=0$ or $\rank(\vertbundle)\leq1$.

\item
For $V_1,V_2\in\Gamma\vertbundle$, the vertical derivations
$D_i=\nabla_{\iota(V_i)}\in\Gamma\so(\Aroid)$ satisfy
\[
\curv^{\widehat g}(\nabla_{\iota(V_1)},\nabla_{\iota(V_2)})
=2\Skew\bigl((\II_{V_2})^\trans\II_{V_1}\bigr),
\]
where
\[
\II_{V_i}=\II(V_i,\,\cdot)
=\frac12\anchor\circ
\bigl(\ad^\madj_{\iota(V_i)}
+\ad^\madj_{\,\cdot}\iota(V_i)\bigr)\circ\iota
:\vertbundle\longrightarrow T\Base,
\qquad i=1,2,
\]
and $\II:\vertbundle\oplus\vertbundle\to T\Base$ is the second
fundamental form of $\iota:\vertbundle\hookrightarrow\Aroid$.
\end{enumerate}
\end{theorem}
The proof is given in Section~\ref{sec:proofs.of.main.theorems}.

\subsection{O'Neill's formulas}
Combining Theorems~\ref{th:algebroid.ONeill} and
\ref{th:curvature.maps.between.derivations} yields O'Neill's formulas for
$\curv^\nabla$, $\curv^{\widehat\nabla}$, and
$\curv^{\widetilde\nabla}$.  Each successive line below evaluates the
preceding identity on one further vector; the colours identify the
corresponding terms.
\newcommand{\lcolor}{red}
\newcommand{\midcolor}{blue}
\newcommand{\rcolor}{black}
\begin{corollary} \label{cor:algebroid.Oneill.formulas}
Let $\Aroid$ be a \RSLA{}  such that
\begin{equation*}
\begin{tikzcd}
\Der_{\so}(\vertbundle) 
& \Der_{\mf{so}}(\Aroid) \ar[r, "g"] \ar[l, "\widehat{g}", swap]
& \Der_{\so}(T\Base) \\
\vertbundle \ar[r, "\iota", hook] \ar[u, "\widehat{\nabla}"] 
& \Aroid \ar[r, "\anchor", two heads] \ar[u, "\nabla"]  \ar[l, "\omega", bend left=25, pos = 0.55]
&T\Base. \ar[l, "\gamma",bend left =25, pos = 0.45] \ar[u, "\widetilde{\nabla}"]
\end{tikzcd}
\end{equation*}
Then the curvature tensors
$\curv^{\nabla}$ and $\curv^{\widetilde{\nabla}}$ of the Levi-Civita connections $\nabla$ and $ \widetilde{\nabla}$ on $\Aroid$ and $T\Base$, respectively, are for $S_i\in T\Base$ related by
\[
\curv^{\widetilde{\nabla}} = {\color{\lcolor} g\circ \gamma^* \curv^{\nabla} }   + {\color{\midcolor} (\nabla \circ \gamma )^* \curv^g } +{\color{\rcolor} g\circ \nabla \circ \curv^\gamma} ,
\]
\[
\curv^{\widetilde{\nabla}}(S_1, S_2) = 
{\color{\lcolor}\anchor\circ \curv^\nabla(\gamma(S_1), \gamma(S_2))\circ \gamma }
+ \frac{1}{4}\left( 
{\color{\midcolor}(\curv^{\gamma}_{S_2})^\trans \curv^\gamma_{S_1} - (\curv^{\gamma}_{S_1})^\trans \curv^\gamma_{S_2}  }
+ {\color{\rcolor}2 (\curv^\gamma_{\cdot})^\trans \curv^{\gamma}_{S_1}S_2 }  \right),
\]
\[
\resizebox{\textwidth}{!}{$\displaystyle
\curv^{\widetilde{\nabla}}(S_1, S_2)S_3 = 
{\color{\lcolor}\anchor \curv^\nabla(\gamma(S_1), \gamma(S_2))\gamma(S_3) }
+ \frac{1}{4}\left( 
{\color{\midcolor}(\curv^{\gamma}_{S_2})^\trans \curv^\gamma_{S_1}S_3 - (\curv^{\gamma}_{S_1})^\trans \curv^\gamma_{S_2}S_3  }
+ {\color{\rcolor}2(\curv^{\gamma}_{S_3})^\trans \curv^\gamma_{S_1}S_2} \right),
$}%
\]
\[
\resizebox{\textwidth}{!}{$\displaystyle
\ip{\curv^{\widetilde{\nabla}}(S_1, S_2)S_3}{S_4} = 
{\color{\lcolor} \ip{\anchor \curv^\nabla(\gamma(S_1), \gamma(S_2))\gamma(S_3) }{S_4} }
+ \frac{1}{4} \left( 
{\color{\midcolor}\ip{\curv^{\gamma}_{S_1}S_3}{\curv^{\gamma}_{S_2}S_4} -\ip{\curv^{\gamma}_{S_2}S_3}{\curv^{\gamma}_{S_1}S_4} }
+{\color{\rcolor} 2 \ip{\curv^{\gamma}_{S_1}S_2}{\curv^{\gamma}_{S_3}S_4} }\right).
$}%
\]

Similarly, the curvature tensors
$\curv^{\nabla}$ and $\curv^{\widehat{\nabla}}$ of the Levi-Civita connections $\nabla$ and $\widehat{\nabla}$ on $\Aroid$ and $\vertbundle$, respectively, are for $V_i\in \vertbundle$ related by
\[
\curv^{\widehat{\nabla}} = {\color{\lcolor} \widehat{g}\circ \iota^* \curv^{\nabla} }  + {\color{\midcolor} (\nabla \circ \iota )^* \curv^{\widehat{g}} } ,
\]
\[
\curv^{\widehat{\nabla}}(V_1, V_2) = 
{\color{\lcolor}\omega\circ \curv^\nabla(\iota(V_1), \iota(V_2))\circ\iota }
+{\color{\midcolor} (\II_{V_2})^\trans \II_{V_1} - (\II_{V_1})^\trans \II_{V_2} },
\]
\[
\curv^{\widehat{\nabla}}(V_1, V_2)V_3 = 
{\color{\lcolor}\omega \curv^\nabla(\iota(V_1), \iota(V_2))\iota(V_3) }
+ {\color{\midcolor}  (\II_{V_2})^\trans \II_{V_1}V_3 - (\II_{V_1})^\trans \II_{V_2}V_3 } ,
\]
\[
\ip{\curv^{\widehat{\nabla}}(V_1, V_2)V_3}{V_4} = 
{\color{\lcolor} \ip{\omega \curv^\nabla(\iota(V_1), \iota(V_2))\iota(V_3) }{V_4} }
+ {\color{\midcolor} \ip{\II_{V_1}V_3 }{\II_{V_2} V_4} - \ip{\II_{V_2}V_3 }{\II_{V_1}V_4} } .
\]
\end{corollary}

For $S_1,S_2\in\Gamma T\Base$, the term $(\nabla\circ\gamma)^*\curv^g(S_1,S_2)$ measures how the commutator of the horizontally lifted covariant derivatives $\nabla_{\gamma(S_1)}$ and $\nabla_{\gamma(S_2)}$ changes under projection by $g$.  Theorem~\ref{th:curvature.maps.between.derivations} shows that, for the Levi-Civita connection, this contribution is determined by the curvature of the horizontal splitting $\gamma$; it therefore introduces no independent invariant in the Riemannian setting.

\section{Finite-dimensional applications}\label{sec:examples}

\subsection{Homogeneous spaces and action algebroid metrics}\label{ex:homogeneous.space}
Let $G$ be a Lie group and $H\subset G$ a closed subgroup, with Lie algebras $\mf h\subset\mf g$.  The homogeneous space $G/H$ is the base of the principal bundle $G(G/H,H,\pi)$.  Its action Lie algebroid and the Atiyah algebroid of this principal bundle occur in the respective short exact sequences
\[
\begin{tikzcd}
\Ad_{G/H}\mf h \ar[r,"\iota",hook] &
\mf g\ltimes G/H \ar[r,"\anchor",two heads] &
T(G/H),
\end{tikzcd}
\qquad
\begin{tikzcd}
{\textstyle\frac{G\times\mf h}{H}} \ar[r,"\iota",hook] &
{\textstyle\frac{TG}{H}} \ar[r,"d\pi",two heads] &
T(G/H).
\end{tikzcd}
\]
Using the left trivialization ${\textstyle\frac{TG}{H}}\simeq{\textstyle\frac{G\times\mf g}{H}}$, the Lie algebroid isomorphism
\[
{\textstyle\frac{TG}{H}}\longrightarrow\mf g\ltimes G/H,
\qquad [g,X]\longmapsto(\Ad_gX,gH),
\]
identifies these sequences.  Suppose that a reductive complement has been chosen:
\[
\mf g=\mf h\oplus\mf m,
\qquad \Ad_H\mf m\subset\mf m.
\]
For the action Lie algebroid $\mf g\ltimes G/H$, define the subbundles $\Ad_{G/H}\mf h$ and $\Ad_{G/H}\mf m$ by
\[
(\Ad_{G/H}\mf h)_{gH}=\Ad_g\mf h,
\qquad
(\Ad_{G/H}\mf m)_{gH}=\Ad_g\mf m.
\]
These are well defined by $\Ad_H$-invariance.  The first is the vertical bundle, while the second determines the reductive horizontal splitting used below.
\begin{proposition}\label{prop:reductive.lie.algebroid}
The chosen reductive complement determines a canonical right splitting
\[
\gamma:T(G/H)\longrightarrow\Ad_{G/H}\mf m\subset\mf g\ltimes G/H
\]
and an associated vector-bundle decomposition
\[
\mf g\ltimes G/H
=\Ad_{G/H}\mf h\oplus\Ad_{G/H}\mf m.
\]
If $\{V_i\}\cup\{W_a\}$ is a basis adapted to $\mf g=\mf h\oplus\mf m$ and $\sigma:U\subset G/H\to G$ is a local section, then
\[
\bar V_i(gH)=\Ad_{\sigma(gH)}V_i,
\qquad
\bar W_a(gH)=\Ad_{\sigma(gH)}W_a
\]
is a local frame adapted to the vertical-horizontal decomposition.  Its brackets satisfy
\[
[\bar V_i,\bar V_j]_{\mf g\ltimes G/H}
=\Ad_{\sigma(gH)}[V_i,V_j]_{\mf h},
\qquad
[\bar W_a,\bar V_i]_{\mf g\ltimes G/H}
=\iota\circ\omega\bigl(\anchor\bar W_a\cdot\bar V_i\bigr),
\]
\[
[\bar W_a,\bar W_b]_{\mf g\ltimes G/H}
=\Ad_{\sigma(gH)}[W_a,W_b]_{\mf g}
+\anchor\bar W_a\cdot\bar W_b
-\anchor\bar W_b\cdot\bar W_a.
\]
\end{proposition}
\begin{proof}
The $\Ad_H$-invariance of $\mf h$ and $\mf m$ makes both subbundles well defined.  Since the kernel of the anchor at $gH$ is $\Ad_g\mf h$, its restriction to $\Ad_g\mf m$ is an isomorphism onto $T_{gH}(G/H)$ and defines $\gamma$.  The bracket identities follow from the action Lie algebroid bracket; for the mixed identity, use that the Lie algebra bracket in $\mf g$ satisfies $[\mf m,\mf h]_{\mf g}\subset\mf m$ and that the Lie algebroid bracket of sections satisfies $[\Gamma(\Ad_{G/H}\mf m),\Gamma(\Ad_{G/H}\mf h)]_{\mf g\ltimes G/H}\subset\Gamma(\Ad_{G/H}\mf h)$.
\end{proof}

A base-independent metric on an action Lie algebroid provides a useful limiting case in which its curvature reduces to that of the acting group.
\begin{theorem}\label{th:constant.curvature}
Let $G$ act on $\Base$.  Equip the action Lie algebroid $\mf g\ltimes\Base$ with the metric induced by a right-invariant metric on $G$:
\[
\ip{X}{Y}(x)=\ip{X}{Y}_G(\Id),
\qquad X,Y\in\mf g\ltimes\set{x},\quad x\in\Base.
\]
Then, under the natural trivialization, the curvature $\curv^\nabla$ of $\mf g\ltimes\Base$ is constant over $\Base$ and equal to the Riemannian curvature $\curv^G$ of $G$.
\end{theorem}
By the correspondence between Lie algebroid metrics and right-invariant source-fibre metrics on its corresponding Lie groupoid, the induced metric on $G\ltimes\Base$ has the same constant curvature $\curv^G$ along its source fibres.
\begin{proof}
The action Lie algebroid $\mf g\ltimes\Base$ is a trivial vector bundle; hence an orthonormal basis of $\mf g$ determines a constant global orthonormal frame.  In this frame the metric and bracket coefficients are those of $\mf g$, while their derivatives along the anchor vanish.  The Koszul formula therefore gives the same connection coefficients as the right-invariant Levi-Civita connection on $G$, and the curvature identity follows.
\end{proof}
We illustrate the use of Theorem~\ref{th:constant.curvature} in Section~\ref{ex:so3.curvatures}.

\subsection{Rotations of the round sphere}\label{ex:so3}\label{ex:so(3).reductive.decomposition}\label{ex:so3.curvatures}
This familiar example illustrates how an action Lie groupoid and its Lie
algebroid encode the geometry of a homogeneous space, and how the three terms in
Theorem~\ref{th:algebroid.ONeill} recover the curvature of the round
sphere as $1/4+1/4+1/2=1$.
The transitive
action of $\SO(3)$ on the round sphere defines the action Lie groupoid
$\SO(3)\ltimes\mbb S^2\rightrightarrows\mbb S^2$, with
\[
\src(A,x)=x,
\qquad
\trg(A,x)=Ax,
\qquad
(B,Ax)(A,x)=(BA,x).
\]
Its isotropy group at $x$ consists of the rotations about the axis through
$x$ and is isomorphic to $\SO(2)$.  A tangent vector to the source fibre at
the unit $(\Id,x)$ is identified with $(X,0)$, for $X\in\so(3)$.  The anchor
of the associated action Lie algebroid $\so(3)\ltimes\mbb S^2$ is therefore
induced by the differential of the target map:
\[
\anchor(X,x)
=d\trg_{(\Id,x)}(X,0)
=\frac{d}{dt}\Big|_{0}\exp(tX)x
=Xx.
\]
Identify $\so(3)$ with $\mbb R^3$ by assigning to $u\in\mbb R^3$ the
infinitesimal rotation $v\mapsto u\times v$.  In this notation,
\[
\anchor(u,x)=u\times x,
\]
and sections $u,v:\mbb S^2\to\mbb R^3$ carry the bracket
\[
[u,v]_{\so(3)\ltimes\mbb S^2}(x)
=-u(x)\times v(x)
+d v_x\bigl(u(x)\times x\bigr)
-d u_x\bigl(v(x)\times x\bigr).
\]
Here the minus sign comes from defining the Lie algebra bracket by
right-invariant vector fields, while the last two terms differentiate the
$\so(3)$-valued sections along their anchor fields.  The stabilizer algebras form
the vertical line bundle, giving the short exact sequence
\[
\begin{tikzcd}
\so(2)_{\mbb S^2}\ar[r,"\iota",hook]
&\so(3)\ltimes\mbb S^2\ar[r,"\anchor",two heads]
&T\mbb S^2.
\end{tikzcd}
\]

Write $x=(x_1,x_2,x_3)\in\mbb S^2$, let
$U=\{x\in\mbb S^2:x_3\neq\pm1\}$, and set
\begin{equation}\label{eq:S2.frame}
S_1=\frac{1}{\sqrt{1-x_3^2}}
\begin{pmatrix}-x_2\\x_1\\0\end{pmatrix},
\qquad
S_2=\frac{1}{\sqrt{1-x_3^2}}
\begin{pmatrix}-x_1x_3\\-x_2x_3\\1-x_3^2\end{pmatrix}.
\end{equation}
The vector fields $S_1,S_2$ form an oriented orthonormal frame for the round
metric on $\mbb S^2|_U$.

The splitting used below is the one induced by the reductive decomposition
from Proposition~\ref{prop:reductive.lie.algebroid}.
Let $e_1,e_2,e_3$ be the standard oriented basis of
$\mbb R^3\simeq\so(3)$.  At $e_1\in\mbb S^2$,
\[
\so(3)=\mf h\oplus\mf m,
\qquad
\mf h=\operatorname{span}\{e_1\},
\qquad
\mf m=\operatorname{span}\{e_3,-e_2\},
\]
is a reductive decomposition, with $\mf h$ the stabilizer algebra.  A local
section $\sigma:U\to\SO(3)$ of $A\mapsto Ae_1$ is
\[
\sigma(x)=\bigl(x\ \ S_1(x)\ \ S_2(x)\bigr)
=
\left(\begin{smallmatrix}
x_1&-\dfrac{x_2}{\sqrt{1-x_3^2}}&-\dfrac{x_1x_3}{\sqrt{1-x_3^2}}\\
x_2& \dfrac{x_1}{\sqrt{1-x_3^2}}&-\dfrac{x_2x_3}{\sqrt{1-x_3^2}}\\
x_3&0&\sqrt{1-x_3^2}
\end{smallmatrix}
\right).
\]
Since $\Ad_Au=Au$ under the cross-product identification, the adapted frame
from Proposition~\ref{prop:reductive.lie.algebroid} is
\begin{equation}\label{eq:so3.frame}
V_0(x)=\Ad_{\sigma(x)}e_1=x,\quad H_1(x)=\Ad_{\sigma(x)}e_3=S_2(x),\quad H_2(x)=\Ad_{\sigma(x)}(-e_2)=-S_1(x).
\end{equation}

The corresponding splitting is
\[
\gamma_x:T_x\mbb S^2\longrightarrow
(\so(3)\ltimes\mbb S^2)_x,
\qquad
\gamma_x(S)=x\times S.
\]
Indeed, $(x\times S)\times x=S$ for $S\in T_x\mbb S^2$, so that
$\anchor\circ\gamma=\Id$; in particular, $H_i=\gamma(S_i)$ for $i=1,2$.
The action-algebroid bracket gives
\[
[H_1,H_2]_{\so(3)\ltimes\mbb S^2}
=V_0-\frac{x_3}{\sqrt{1-x_3^2}}H_1,
\qquad
[V_0,H_i]_{\so(3)\ltimes\mbb S^2}=0,
\quad i=1,2.
\]
Since $[S_1,S_2]_{T\mbb S^2}
=-\frac{x_3}{\sqrt{1-x_3^2}}S_1$, the curvature of the splitting is
\begin{equation}\label{eq:so3.splitting.curvature}
\curv^\gamma(S_1,S_2)
=\gamma([S_1,S_2]_{T\mbb S^2})
-[\gamma(S_1),\gamma(S_2)]_{\so(3)\ltimes\mbb S^2}
=-V_0.
\end{equation}

\begin{figure}
\centering
\includegraphics[width=0.5\textwidth]{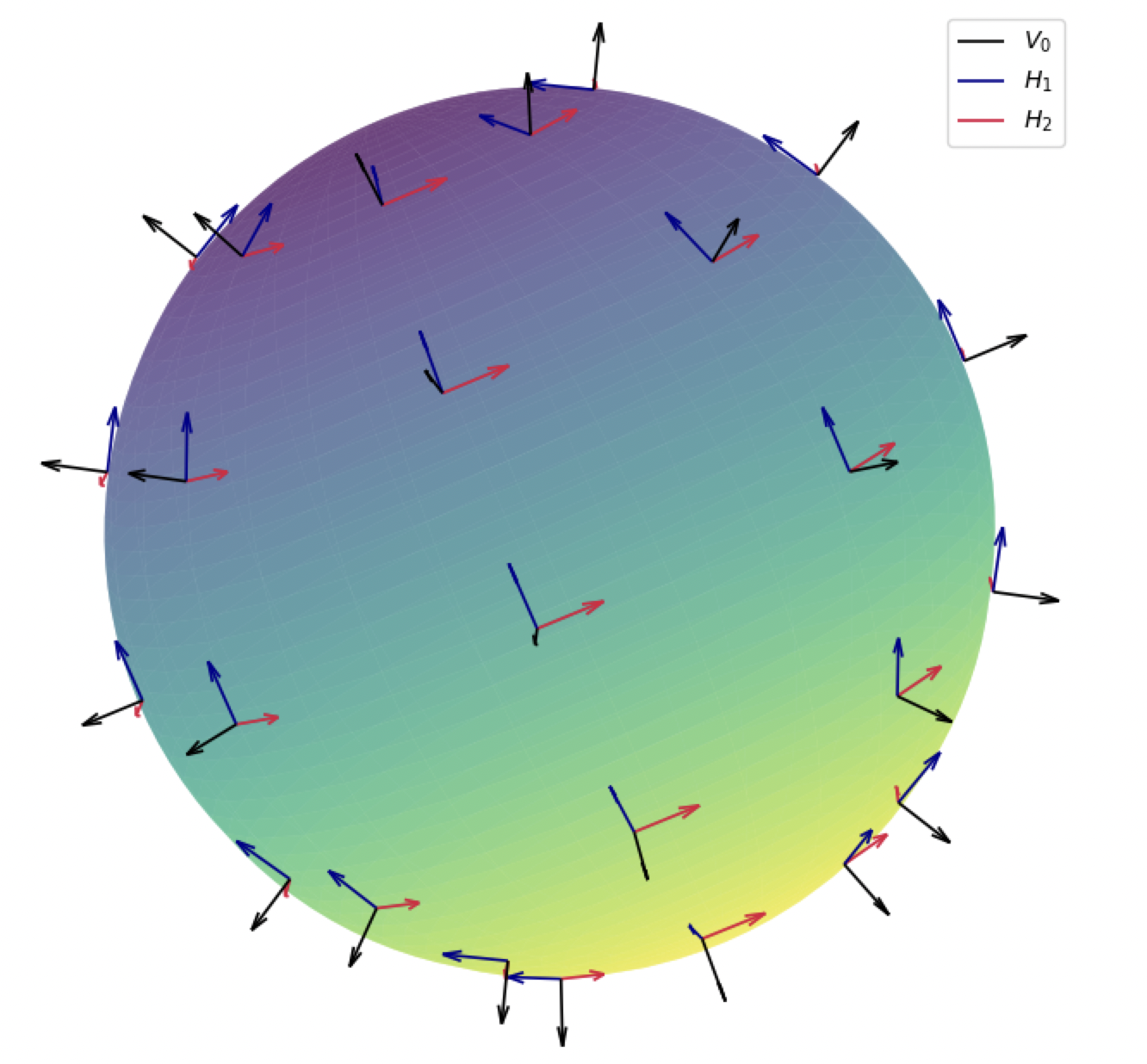}
\caption{The adapted frame of $\so(3)\ltimes\mbb S^2$.}
\label{fig:so3.action.algebroid}
\end{figure}

Under the identification $\so(3)\simeq\mbb R^3$ used above, equip
$\so(3)\ltimes\mbb S^2$ with the constant Euclidean metric
\[
\ip{u}{v}(x)=\ip{u}{v}_{\mbb R^3},
\qquad
u,v\in(\so(3)\ltimes\mbb S^2)_x,\quad x\in\mbb S^2.
\]
The adapted algebroid frame $V_0,H_1,H_2$ in~\eqref{eq:so3.frame} is then
orthonormal, and the anchor maps its horizontal part isometrically onto the
round tangent frame $S_1,S_2$ in~\eqref{eq:S2.frame}.  Hence
$\so(3)\ltimes\mbb S^2$ is a \RSLA{}.  The orthogonal splitting $\gamma$
induced by the constant Euclidean metric coincides, in this case, with the
splitting induced by the reductive decomposition above.

The Levi-Civita connections $\nabla$ and $\widetilde\nabla$, together
with the anchor-preserving map $g$ and the splitting $\gamma$, fit with
the three anchor sequences as follows:
\[
\begin{tikzcd}[column sep=large,row sep=large]
\so(2)_{\mbb S^2}
 \ar[r,"\iota",hook]
&
\so(3)\ltimes\mbb S^2
 \ar[d,"\nabla"]
 \ar[dr,"\anchor",two heads]
&
\\
\so(\so(3)\ltimes\mbb S^2)
 \ar[r,"\iota_{\Der}",hook]
&
\Der_{\so}(\so(3)\ltimes\mbb S^2)
 \ar[d,"g"]
 \ar[r,"\anchor_{\Der}",two heads]
&
T\mbb S^2
 \ar[ul,"\gamma",bend right=25,swap]
 \ar[dl,"\widetilde\nabla"',bend left=25, swap]
\\
\so(T\mbb S^2)
 \ar[r,"\iota_{\Der}",hook]
&
\Der_{\so}(T\mbb S^2)
 \ar[ur,"\anchor_{\Der}",two heads]
&
\end{tikzcd}
\]
Regard the basis elements $e_1,e_2,e_3$ introduced above as constant global
sections of $\so(3)\ltimes\mbb S^2$, and denote their metric-dual sections
by $e^1,e^2,e^3$.  Although the sections $e_i$ are constant, their anchors
are the nonconstant Killing fields
\[
\anchor e_i(x)=e_i\times x,
\qquad x\in\mbb S^2.
\]
Their brackets are
\[
[e_1,e_2]_{\so(3)\ltimes\mbb S^2}=-e_3,\qquad
[e_2,e_3]_{\so(3)\ltimes\mbb S^2}=-e_1,\qquad
[e_3,e_1]_{\so(3)\ltimes\mbb S^2}=-e_2.
\]
For the skew endomorphisms below, write
$e_i\wedge e^j=e_i\otimes e^j-e_j\otimes e^i$.
The Euclidean metric on $\mbb R^3\simeq\so(3)$ is $\Ad$-invariant, so
$\ad^\madj_{e_i}=\ad_{e_i}$.  Lemma~\ref{lem:levi.civita.formula} gives
\begin{equation}\label{eq:so3.levi.civita}
\nabla_{e_1}=\anchor e_1\oplus\frac12e_2\wedge e^3,\qquad
\nabla_{e_2}=\anchor e_2\oplus\frac12e_3\wedge e^1,\qquad
\nabla_{e_3}=\anchor e_3\oplus\frac12e_1\wedge e^2.
\end{equation}

The curvature is now obtained directly as a skew operator.  From the
brackets above and~\eqref{eq:so3.levi.civita},
\[
[\nabla_{e_1},\nabla_{e_2}]_{\Der}
=-\anchor e_3\oplus\frac14e_2\wedge e^1,
\qquad
\nabla_{-e_3}=-\anchor e_3\oplus\frac12e_2\wedge e^1.
\]
Definition~\ref{def:curvature} therefore gives
\[
\curv^\nabla(e_1,e_2)
=\nabla_{-e_3}-[\nabla_{e_1},\nabla_{e_2}]_{\Der}
=\frac14e_2\wedge e^1.
\]
By cyclic permutation,
\begin{equation}\label{eq:so3.algebroid.curvature}
\curv^\nabla(e_i,e_j)=\frac14e_j\wedge e^i,
\qquad i\neq j.
\end{equation}
Here we used that $\curv^\nabla$ takes values in
$\so(\so(3)\ltimes\mbb S^2)$ and hence may be computed directly in the
skew basis $e_i\wedge e^j$.  In particular,
$C^\nabla(e_i,e_j)=\frac14$ for $i\neq j$, consistently
with Theorem~\ref{th:constant.curvature} and the curvature of the source
fibres of $\SO(3)\ltimes\mbb S^2$.

Theorem~\ref{th:algebroid.ONeill} gives the factorization
\begin{equation}\label{eq:so3.base.curvature.factorization}
\curv^{\widetilde\nabla}
=(\nabla\circ\gamma)^*\curv^g
+g\circ\gamma^*\curv^\nabla
+g\circ\nabla\circ\curv^\gamma.
\end{equation}
As displayed above, $\widetilde\nabla=g\circ\nabla\circ\gamma$.  Let
$V^0,H^1,H^2$ be the coframe dual to $V_0,H_1,H_2$, and let $S^1,S^2$ be
the coframe dual to the frame $S_1,S_2$ of $T\mbb S^2|_U$.  By
Theorem~\ref{th:curvature.maps.between.derivations}
(\ref{th:horisontal.derivation.curv.formula}) and
\eqref{eq:so3.splitting.curvature},
\[
\curv^\gamma_{S_1}=-V_0\otimes S^2,
\qquad
\curv^\gamma_{S_2}=V_0\otimes S^1.
\]
Consequently,
\[
\bigl((\nabla\circ\gamma)^*\curv^g\bigr)(S_1,S_2)
=\frac12\Skew\bigl((\curv^\gamma_{S_2})^\trans
\curv^\gamma_{S_1}\bigr)
=\frac14S_2\wedge S^1.
\]
Equation~\eqref{eq:so3.algebroid.curvature} gives
$\curv^\nabla(H_1,H_2)=\frac14H_2\wedge H^1$, and hence
\[
\bigl(g\circ\gamma^*\curv^\nabla\bigr)(S_1,S_2)
=g\circ\curv^\nabla\bigl(\gamma(S_1),\gamma(S_2)\bigr)
=\frac14g(H_2\wedge H^1)
=\frac14S_2\wedge S^1.
\]
Finally, since $V_0=x_1e_1+x_2e_2+x_3e_3$, the
$C^\infty(\mbb S^2)$-linearity of $\nabla$ in its first argument and
$\anchor V_0=0$, together with~\eqref{eq:so3.levi.civita}, yield
\[
\nabla_{V_0}
=x_1\nabla_{e_1}+x_2\nabla_{e_2}+x_3\nabla_{e_3}
=\frac12\bigl(x_1e_2\wedge e^3+x_2e_3\wedge e^1
+x_3e_1\wedge e^2\bigr)
=-\frac12H_2\wedge H^1.
\]
Together with~\eqref{eq:so3.splitting.curvature}, this gives
\[
\bigl(g\circ\nabla\circ\curv^\gamma\bigr)(S_1,S_2)
=g(\nabla_{-V_0})
=\frac12g(H_2\wedge H^1)
=\frac12S_2\wedge S^1.
\]
Substitution in~\eqref{eq:so3.base.curvature.factorization} yields
\[
\curv^{\widetilde\nabla}(S_1,S_2)
=\left(\frac14+\frac14+\frac12\right)S_2\wedge S^1
=S_2\wedge S^1.
\]
Consequently,
\[
C^{\mbb S^2}(S_1,S_2)
=\ip{\curv^{\widetilde\nabla}(S_1,S_2)S_1}{S_2}
=1.
\]
Thus the three terms in the Lie-algebroid factorization recover the round
curvature as $1/4+1/4+1/2=1$.

\subsection{A nonconstant metric: the Earth ellipsoid}\label{sec:earth.ellipsoid}

In this example, we illustrate the preceding framework through the WGS84
ellipsoidal model of the Earth.  As in the round case, we use the rotational
action Lie algebroid $\so(3)\ltimes\mbb S^2$, now equipped with a non-constant
metric.  The computation recovers the Gaussian curvature of the ellipsoid
from the three curvature contributions in Theorem~\ref{th:algebroid.ONeill},
even though the full $\SO(3)$-action is no longer isometric.

Let $0\leq \ecc<1$ and consider the normalized oblate ellipsoid
\[
E
=\set{y=(y_1,y_2,y_3)\in\mbb R^3\mid
y_1^2+y_2^2+\frac{y_3^2}{1-\ecc^2}=1},
\]
equipped with the induced Euclidean metric $\ip{\cdot}{\cdot}_E$.  The WGS84 value is
$\ecc\simeq0.0818$~\cite{national_geospatial-intelligence_agency_department_2014}.
The map
\[
F:\mbb S^2\longrightarrow E,
\qquad
F(x_1,x_2,x_3)=(x_1,x_2,\sqrt{1-\ecc^2}\,x_3),
\]
pulls the induced metric $\ip{\cdot}{\cdot}_E$ back to
\begin{equation}\label{eq:ellipsoidal.metric}
\ip{\cdot}{\cdot}_{\mbb S^2}
=F^*\ip{\cdot}{\cdot}_E
=dx_1^2+dx_2^2+(1-\ecc^2)dx_3^2
\quad\text{on }T\mbb S^2.
\end{equation}
For $\ecc\neq0$, the full $\SO(3)$-action is no longer isometric, but the
same action algebroid $\so(3)\ltimes\mbb S^2$ may still be equipped with
a metric for which its anchor is a Riemannian submersion onto
$(\mbb S^2,\ip{\cdot}{\cdot}_{\mbb S^2})$.

Since the eccentricity of the Earth is small, the correction term $1-\ecc^2\simeq 0.9933$ seems negligible. However, at the scale of the Earth the difference between the radius at the equator and the radius at the rotational poles is 21 \si{km}. Meridian distances close to the equator are scaled by $\sqrt{1-\ecc^2} \simeq 99.7\%$, so e.g. the north-south measurement of Lake Victoria, Africa's largest lake, would be off by 1 \si{km}. 

\begin{figure}
\centering
\includegraphics[width=0.65\textwidth]{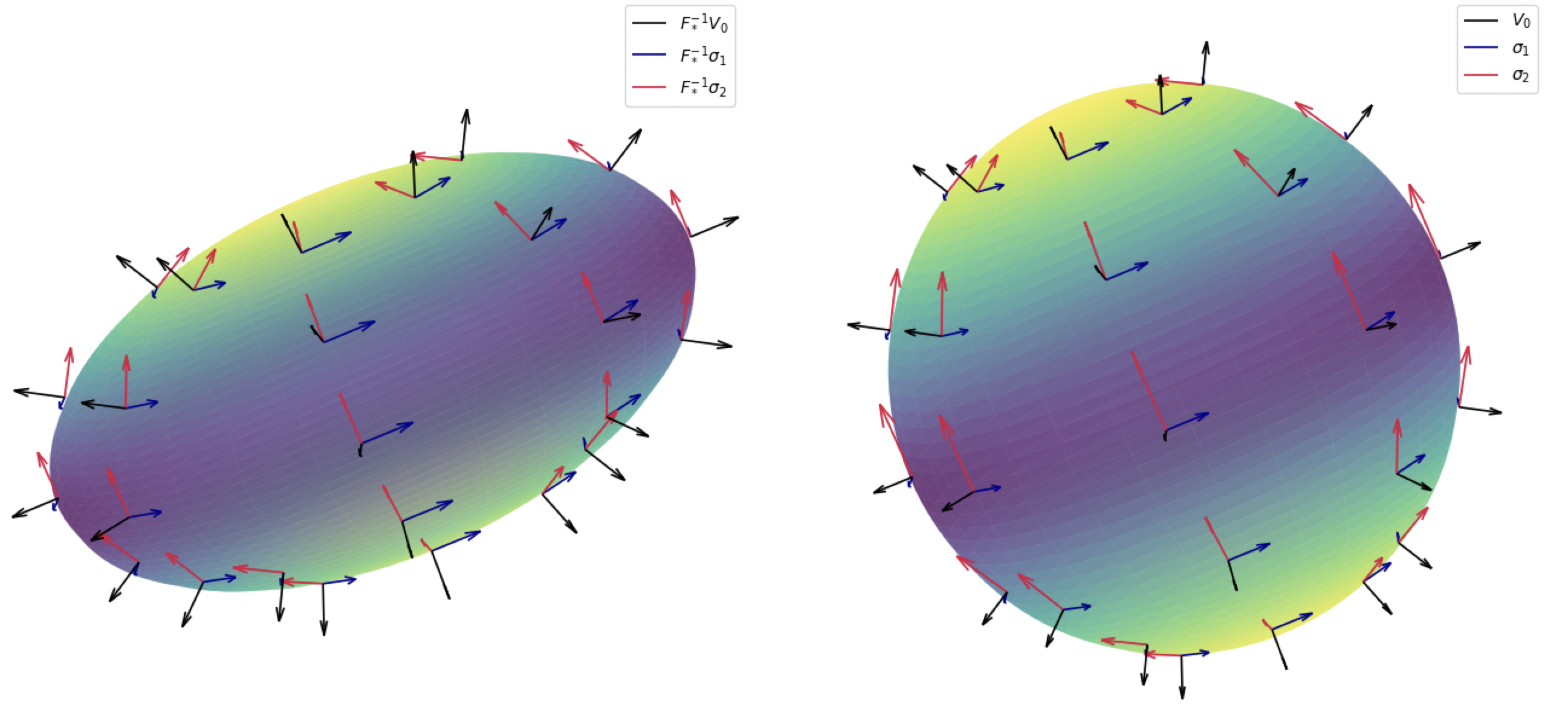}
\caption{Left: the normalized WGS84 ellipsoid $E$, with eccentricity
exaggerated.  Right: $\mbb S^2$, identified with $E$ by $F$, with the adapted
frame $\sigma_1,\sigma_2$ and vertical isotropy field $V_0$.}
\label{fig:earth.frame}
\end{figure}

On the open set $U$ of the preceding example, retain the round frame
$S_1,S_2$ and put
\[
c(x)=\frac{1}{\sqrt{1-\ecc^2(1-x_3^2)}},
\qquad
\sigma_1=S_1,\qquad \sigma_2=cS_2.
\]
Then $\sigma_1,\sigma_2$ is an orthonormal frame for
$\ip{\cdot}{\cdot}_{\mbb S^2}$.  Retain
also the reductive splitting $\gamma$, so that $H_i=\gamma(S_i)$, and set
\[
\eta_1=H_1=\gamma(\sigma_1),
\qquad
\eta_2=cH_2=\gamma(\sigma_2).
\]
We define the algebroid metric $\ip{\cdot}{\cdot}$ by declaring
$V_0,\eta_1,\eta_2$
orthonormal.  Equivalently, if $V^0,\eta^1,\eta^2$ is the dual coframe,
\[
\ip{\cdot}{\cdot}
=V^0\otimes V^0+\eta^1\otimes\eta^1+\eta^2\otimes\eta^2.
\]
Thus $(\so(3)\ltimes\mbb S^2,\ip{\cdot}{\cdot})$ is a \RSLA{},
and it determines a right-invariant source-fibre metric on
$\SO(3)\ltimes\mbb S^2$.

In the orthonormal frame $V_0,\eta_1,\eta_2$, the bracket and the
Levi--Civita connection are
\begin{equation}\label{eq:so3.earth.algebroid.bracket}
[\eta_1, \eta_2]_{\so(3) \ltimes \mbb S^2} = c\big(V_0 - {\textstyle \frac{x_3}{\sqrt{1-x_3^2}} }\eta_1 \big), \quad [V_0, \eta_1]_{\so(3) \ltimes \mbb S^2} = [V_0, \eta_2]_{\so(3) \ltimes \mbb S^2} =0.
\end{equation}
\begin{equation}\label{eq:levi.civita.so3.earth}
\nabla_{V_0} = - {\textstyle \frac{c}{2}} \eta_2\wedge \eta^1, \quad 
\nabla_{\eta_1} = \sigma_1 \oplus \Big( {\textstyle - \frac{c}{2}} \eta_2\wedge V^0 + {\textstyle \frac{cx_3}{\sqrt{1-x_3^2}} } \eta_2\wedge \eta^1 \Big), \quad
\nabla_{\eta_2} = \sigma_2 \oplus {\textstyle  \frac{c}{2}} \eta_1 \wedge V^0 .
\end{equation}
A direct computation from~\eqref{eq:levi.civita.so3.earth} gives
\begin{equation}\label{eq:earth.algebroid.curvature}
\curv^\nabla(\eta_1,\eta_2)
=c^2\left(\frac14-c^2\ecc^2x_3^2\right)\eta_2\wedge\eta^1
+\frac{c^4\ecc^2x_3\sqrt{1-x_3^2}}2\eta_2\wedge V^0.
\end{equation}
In particular, the horizontal sectional curvature of the source fibres is
$c^2(\frac14-c^2\ecc^2x_3^2)$ and is no longer constant.

We now apply Theorem~\ref{th:algebroid.ONeill}.  From
\eqref{eq:so3.splitting.curvature} and tensoriality,
\[
\curv^\gamma(\sigma_1,\sigma_2)=-cV_0,
\qquad
\curv^\gamma_{\sigma_1}=-cV_0\otimes\sigma^2,
\qquad
\curv^\gamma_{\sigma_2}=cV_0\otimes\sigma^1,
\]
where $\sigma^1,\sigma^2$ denotes the coframe dual to
$\sigma_1,\sigma_2$.  Equations~\eqref{eq:levi.civita.so3.earth} and
\eqref{eq:earth.algebroid.curvature}, together with
Theorem~\ref{th:curvature.maps.between.derivations}, give the three
contributions
\begin{align*}
\bigl(g\circ\gamma^*\curv^\nabla\bigr)(\sigma_1,\sigma_2)
&=c^2\left(\frac14-c^2\ecc^2x_3^2\right)
  \sigma_2\wedge\sigma^1,\\
\bigl((\nabla\circ\gamma)^*\curv^g\bigr)(\sigma_1,\sigma_2)
&=\frac{c^2}{4}\sigma_2\wedge\sigma^1,\\
\bigl(g\circ\nabla\circ\curv^\gamma\bigr)(\sigma_1,\sigma_2)
&=g(\nabla_{-cV_0})
=\frac{c^2}{2}\sigma_2\wedge\sigma^1.
\end{align*}
Consequently,
\[
 \curv^{\widetilde\nabla}(\sigma_1,\sigma_2)
=c^2\left(
\frac14-c^2\ecc^2x_3^2+\frac14+\frac12\right)
\sigma_2\wedge\sigma^1
=\frac{1-\ecc^2}
{\bigl(1-\ecc^2(1-x_3^2)\bigr)^2}
\sigma_2\wedge\sigma^1.
\]
Thus the sectional curvature of
$(\mbb S^2,\ip{\cdot}{\cdot}_{\mbb S^2})$ is
\[
C^{\mbb S^2}(x)
=\frac{1-\ecc^2}
{\bigl(1-\ecc^2(1-x_3^2)\bigr)^2};
\]
it has the values $(1-\ecc^2)^{-1}$ at the equator and $1-\ecc^2$ at
the poles.  This is the classical Gaussian curvature of the normalized
oblate ellipsoid.  The point here is its recovery from three curvature
contributions on the rotational action algebroid, despite the failure of
the full $\SO(3)$-action to preserve $\ip{\cdot}{\cdot}_{\mbb S^2}$.

\subsection{Non-abelian isotropy over the projective line}\label{ex:RP1}
This example illustrates the fibrewise statement of
Proposition~\ref{prop:adbundle.levi.civita} for a non-abelian isotropy
bundle, whose metric can be chosen to realise any prescribed smooth
negative sectional-curvature function on the base.
Consider the transitive action of
\[
\PSL
=\set{\left.\begin{pmatrix}a&b\\c&d\end{pmatrix}
\ \right|\ a,b,c,d\in\mbb R,\ ad-bc=1}/\set{\pm I}
\]
on
\[
\RP=(\mbb R^2\setminus\set{0})/\mathord\sim,
\qquad
\begin{pmatrix}a&b\\c&d\end{pmatrix}\cdot[x:y]
=[ax+by:cx+dy].
\]
Its Lie algebra is
\[
\spl=\set{\left.\begin{pmatrix}a&b\\c&-a\end{pmatrix}
\ \right|\ a,b,c\in\mbb R},
\qquad [X,Y]_{\spl}=-(XY-YX),
\]
and the infinitesimal action defines the transitive action algebroid
\[
\begin{tikzcd}
\bor_{\RP}\ar[r,"\iota",hook]&
\spl\ltimes\RP\ar[r,"\anchor",two heads]&T\RP.
\end{tikzcd}
\]
Indeed,
\[
\anchor
\begin{pmatrix}a&b\\c&-a\end{pmatrix}[x:y]
=(ax+by,cx-ay)\mod\mbb R(x,y),
\]
and the isotropy algebra at $[1:0]$ is the Borel algebra
\[
\bor=\set{\left.\begin{pmatrix}\alpha&\beta\\0&-\alpha\end{pmatrix}
\ \right|\ \alpha,\beta\in\mbb R}.
\]
The other isotropy algebras are its conjugates; together they form
$\bor_{\RP}$, the Lie algebroid of the stabilizer groups $B_{[x:y]}$.

Equip $\spl\ltimes\RP$ with the constant Frobenius metric
\[
\ip{X}{Y}=\tr(X^\trans Y).
\]
Its restriction to $\bor_{\RP}$ admits the global orthonormal frame
\[
V_1[x:y]={\textstyle\frac{1}{x^2+y^2}}
\begin{pmatrix}xy&-x^2\\y^2&-xy\end{pmatrix},
\qquad
V_2[x:y]={\textstyle\frac{1}{\sqrt2(x^2+y^2)}}
\begin{pmatrix}x^2-y^2&2xy\\2xy&y^2-x^2\end{pmatrix}.
\]
Since the anchor of $\bor_{\RP}$ vanishes, its bracket is fibrewise, and
\[
[V_1,V_2]_{\bor_{\RP}}=\sqrt2\,V_1,
\qquad
\widehat\nabla_{V_1}=\sqrt2\,V_1\wedge V^2,
\qquad
\widehat\nabla_{V_2}=0.
\]
Proposition~\ref{prop:adbundle.levi.civita} therefore identifies
$\widehat\nabla|_{[x:y]}$ with the Levi-Civita connection of the Borel
stabilizer $B_{[x:y]}$, and
\[
C^{B_{[x:y]}}(\widetilde V_1,\widetilde V_2)
=C^{\widehat\nabla}(V_1,V_2)=-2.
\]

More generally, let $f_1,f_2\in C^\infty(\RP,\mbb R_{>0})$ and define
\[
\ip{\cdot}{\cdot}
=f_1^{-2}V^1\otimes V^1+f_2^{-2}V^2\otimes V^2.
\]
Then $W_1=f_1V_1$ and $W_2=f_2V_2$ form an orthonormal frame.  Since
$\bor_{\RP}$ is totally intransitive,
\[
[W_1,W_2]_{\bor_{\RP}}=\sqrt2\,f_2W_1,
\qquad
C^{B_{[x:y]}}(\widetilde W_1,\widetilde W_2)
=C^{\widehat\nabla}(W_1,W_2)[x:y]
=-2\bigl(f_2[x:y]\bigr)^2.
\]
Thus the Levi-Civita connection on $\bor_{\RP}$ is precisely a smooth
family of Levi-Civita connections on the isotropy groups.  Moreover, the
sectional curvature of this family is negative, independent of $f_1$, and
can be prescribed within the negative range by a suitable choice of $f_2$.

\section{Diffeomorphisms acting on densities on the torus}\label{ex:diffeo.groupoid}
\subsection{The action algebroid and its \texorpdfstring{$L^2$}{L2} metric}
We now apply the preceding theory to the action of the diffeomorphism group
on densities.  This is the principal application of the paper: the vertical
and horizontal parts of one Riemannian submersion Lie algebroid recover,
respectively, Arnold's curvature of the group of volume-preserving
diffeomorphisms and the Wasserstein curvature of the space of densities.

We work with smooth objects in the Fr\'echet category.
The projection $(\varphi,\mu)\mapsto\varphi$ is an isometry from the
source fibre over $\mu$ to $\Difftor$ equipped with the $L^2$ metric
defined by integration against $\mu$.
The Levi-Civita connection on each source fibre is thus obtained by
transporting the connection of Ebin-Marsden~\cite{ebin_groups_1970}
through this isometry.  Let
\[
\torus=\mbb R^2/(2\pi\mbb Z)^2
\]
with its flat metric and standard density $dx^2$, normalized by
$\int_{\torus}dx^2=4\pi^2$.  Denote by $\Denstor$ the smooth positive
densities $\mu$ on $\torus$ satisfying $\int_{\torus}\mu=4\pi^2$.
The push-forward action defines the action groupoid
$\Difftor\ltimes\Denstor\rightrightarrows\Denstor$.  Its isotropy group at
$\mu$ is the group of $\mu$-preserving diffeomorphisms, the configuration
space of an incompressible fluid with mass density $\mu$.
The construction is the single-phase case of the multiphase action groupoid in
\cite{izosimov_geometry_2023}.

The corresponding action Lie algebroid fits into
\[
\begin{tikzcd}
\SVecttor_{\Denstor}\ar[r,"\iota",hook]&
\Vecttor\ltimes\Denstor\ar[r,"\anchor",two heads]&
T\Denstor,
\end{tikzcd}
\qquad
\anchor X(\mu)=-\Lie_X\mu.
\]
The fibre of $\SVecttor_{\Denstor}$ over $\mu$ is the Lie algebra of vector fields divergence-free with respect to $\mu$, 
$\SVecttor_\mu=\set{X\in\Vecttor\mid\Lie_X\mu=0}$.
Equip the action algebroid with the $L^2$ metric
\begin{equation}\label{eq:L2.metric}
\ip{X}{Y}(\mu)=\int_{\torus}X(x)\cdot Y(x)\,\mu(x).
\end{equation}
The orthogonal complement of $\SVecttor_\mu$ consists of gradient
fields; the metric induced on $T\Denstor$ is the Wasserstein metric.

For a Riemannian manifold $M$, \cite{ebin_groups_1970} showed that the
Levi-Civita connection induced on $\Diff(M)$ is right-invariant, even though
the $L^2$ metric is only right-invariant with respect to measure-preserving
diffeomorphisms.  Moreover, at the identity of $\Diff(M)$ this connection
coincides with the Levi-Civita connection on $M$.  It follows that the
Riemannian curvature of $\Diff(M)$ with the $L^2$ metric is completely
determined by the Riemannian curvature of $M$.  Since $\torus$ is flat, the
Riemannian connection
\[
\nabla:\Vecttor\ltimes\Denstor\longrightarrow
\Der_{\so}(\Vecttor\ltimes\Denstor)
\]
is flat, i.e. $\curv^\nabla=0$.  The two identities of
Theorem~\ref{th:algebroid.ONeill} therefore reduce to
\[
\curv^{\widehat\nabla}
=(\nabla\circ\iota)^*\curv^{\widehat g},
\qquad
\curv^{\widetilde\nabla}
=(\nabla\circ\gamma)^*\curv^g
+g\circ\nabla\circ\curv^\gamma.
\]

The first identity expresses the curvature of the isotropy groups through
the second fundamental form of their inclusion in the source fibres and will
recover Arnold's formula~\cite{arnold_sur_1966}.  The second expresses the
curvature of $\Denstor$ through the non-involutivity of the gradient
distribution and yields the explicit torus version of the Wasserstein
calculation in~\cite{otto_geometry_2001,lott_geometric_2007}.

\subsection{Fourier frames}
We introduce a global complex Fourier frame for
$\VecttorC\ltimes\Denstor$.  Write
\begin{equation}\label{eq:index}
\mbf k=(k,k_3)=(k_1,k_2,k_3)\in\mbb Z^2\times\mbb Z_2
\end{equation}
and, for $k\neq0$, set
\begin{equation}\label{eq:const.vector.fields}
a^{(k,0)}=\frac{1}{\norm{k}}
  \bigl(k_1\p_{x_1}+k_2\p_{x_2}\bigr),
\qquad
a^{(k,1)}=\frac{1}{\norm{k}}
  \bigl(k_2\p_{x_1}-k_1\p_{x_2}\bigr).
\end{equation}
The two zero modes are
$a^{(0,0,0)}=i\p_{x_1}$ and $a^{(0,0,1)}=i\p_{x_2}$.
Here $\norm{k}=\sqrt{k_1^2+k_2^2}$ and
$k\times l=k_1l_2-k_2l_1$.  We write
\[
a^{\mbf k}=b^{\mbf k}_1\p_{x_1}+b^{\mbf k}_2\p_{x_2},
\qquad
b^{\mbf k}=(b^{\mbf k}_1,b^{\mbf k}_2),
\]
so that $b^{\mbf k}$ is the coefficient vector of the constant vector field
$a^{\mbf k}$.  Every $X\in\VecttorC$ has a Fourier expansion
\begin{equation}\label{eq:Vecttor}
X(x)=\sum_{\mbf k\in\mbb Z^2\times\mbb Z_2}
e^{ik\cdot x}c_{\mbf k}a^{\mbf k},
\qquad c_{\mbf k}\in\mbb C.
\end{equation}
For $k\neq0$, the modes with $k_3=0$ are gradients and those with $k_3=1$
are divergence-free; the two zero modes are constant vector fields.

Following~\cite{arnold_sur_1966}, extend the $L^2$ metric complex
bilinearly.  Writing $\mu=dx^2/\rho(x)^2$, define
\begin{equation}\label{eq:Vecttor.frame}
X_{\mbf k}\left({\textstyle\frac{dx^2}{\rho(x)^2}}\right)
=\frac{\rho(x)}{2\pi i}e^{ik\cdot x}a^{\mbf k}.
\end{equation}
Since $\int_\torus e^{i(k+l)\cdot x}\,dx^2=4\pi^2\delta_{k+l,0}$,
\[
\ip{X_{\mbf k}}{X_{\mbf l}}(\mu)
=-(a^{\mbf k}\cdot a^{\mbf l})\delta_{k+l,0}
=\delta_{\mbf k+\mbf l,0}.
\]
Thus the nonzero mode $X_{\mbf k}$ is dual to $X_{-\mbf k}$, while each
of the two zero modes has norm one.  Over every $\mu\in\Denstor$, these
modes form a Fourier basis of $\VecttorC$ with the above duality.

At $\mu=dx^2$ the frame is adapted to the horizontal-vertical splitting.
For $k\in\mbb Z^2\setminus\set{0}$, set
\begin{equation}\label{eq:restricted.bases}
S_k=\anchor X_{(k,0)}(dx^2)
=-\frac{\norm{k}}{2\pi}e^{ik\cdot x}\,dx^2,
\qquad
V_k=X_{(k,1)}(dx^2)
=\frac{e^{ik\cdot x}}{2\pi i\norm{k}}
 \bigl(k_2\p_{x_1}-k_1\p_{x_2}\bigr).
\end{equation}
The $S_k$ form a complex Fourier basis of
$T_{dx^2}\Denstor\otimes\mbb C$, and the $V_k$, together with the two
constant vertical modes
\[
X_{(0,0,0)}(dx^2)=\frac{1}{2\pi}\p_{x_1},
\qquad
X_{(0,0,1)}(dx^2)=\frac{1}{2\pi}\p_{x_2},
\]
form a complex Fourier basis of $(\SVecttor_{\mbb C})_{dx^2}$.

The complex Fourier modes are convenient for computing the curvature tensor,
whereas sectional curvature is evaluated on a real orthonormal basis.  Let
\[
\mbb Z^2_+=\set{(m_1,m_2)\in\mbb Z^2
\mid m_1>0\text{, or }m_1=0\text{ and }m_2>0},
\]
so that $\mbb Z^2_+$ contains exactly one element of each pair $\set{m,-m}$
with $m\neq0$.  For $m\in\mbb Z^2_+$, set
\begin{equation}\label{eq:real.density.modes}
S_m^{\Re}=-\frac{1}{\sqrt2}(S_m+S_{-m})
=\frac{\norm{m}}{\sqrt2\pi}\cos(m\cdot x)\,dx^2,
\qquad
S_m^{\Im}=\frac{i}{\sqrt2}(S_m-S_{-m})
=\frac{\norm{m}}{\sqrt2\pi}\sin(m\cdot x)\,dx^2.
\end{equation}
Then
\[
\set{S_m^{\Re},S_m^{\Im}\mid m\in\mbb Z^2_+}
\]
is a real orthonormal basis of $T_{dx^2}\Denstor$.

\subsection{Vertical and horizontal curvature}
\begin{theorem}\label{th:vecttor.hor.vert.curv}
Let $k,l,m,n\in\mbb Z^2\setminus\set{0}$.
\begin{enumerate}
\item Set
$\curv^{\widehat\nabla}_{klmn}
=\ip{\curv^{\widehat\nabla}(V_k,V_l)V_m}{V_n}$
for vertical vectors $\set{V_k}\subset(\SVecttor_{\mbb C})_{dx^2}$.
The Riemannian curvature of the group $\SDifftor$ of measure-preserving
diffeomorphisms of the torus for $\mu=dx^2$ satisfies
\[
\curv^{\widehat\nabla}_{klmn}
=\frac{1}{4\pi^2\norm{k}\norm{l}\norm{m}\norm{n}}
\left(
\frac{(k\times m)^2(l\times n)^2}{\norm{k+m}^2}
-\frac{(l\times m)^2(k\times n)^2}{\norm{l+m}^2}
\right)
\]
for $k+l+m+n=0$.  If $k+l+m+n\neq0$, then
$\curv^{\widehat\nabla}_{klmn}=0$.

\item The horizontal distribution $\gamma(T\Denstor\otimes\mbb C)$ is not
integrable.  For $k+l\neq0$, the curvature of the horizontal lift $\gamma$
over the density $\mu=dx^2$ is given by
\[
\curv^\gamma(S_k,S_l)
=\frac{(k\times l)(k\cdot l)}
 {\pi\norm{k}\norm{l}\norm{k+l}}V_{k+l}.
\]
For the resonant pair $l=-k$, the curvature of the horizontal lift is the
constant vertical vector field
\[
\curv^\gamma(S_k,S_{-k})
=\frac{i}{2\pi^2}
\bigl(k_1\p_{x_1}+k_2\p_{x_2}\bigr).
\]

\item The sectional curvature of the space $\Denstor$ of smooth densities on
the torus endowed with the Wasserstein metric is non-negative for every
$\mu\in\Denstor$.  Moreover, over the density $\mu=dx^2$, let $k,l\in
\mbb Z^2_+$ be distinct and choose
$\widetilde S_j\in\set{S_j^{\Re},S_j^{\Im}}$ for $j=k,l$.  The sectional
curvature of the plane spanned by $\widetilde S_k$ and $\widetilde S_l$ is
\[
C^{\widetilde\nabla}(\widetilde S_k,\widetilde S_l)
=\frac{3(k\times l)^2(k\cdot l)^2}
 {8\pi^2\norm{k}^2\norm{l}^2}
\left(\frac{1}{\norm{k+l}^2}+\frac{1}{\norm{k-l}^2}\right).
\]
The two basis vectors of a fixed frequency $k\in\mbb Z^2_+$ span the plane
of sectional curvature
\[
C^{\widetilde\nabla}(S_k^{\Re},S_k^{\Im})
=\frac{3\norm{k}^2}{4\pi^2}.
\]
\end{enumerate}
\end{theorem}

Theorem~\ref{th:vecttor.hor.vert.curv} should be viewed as a refinement and
unification, using the Lie algebroid machinery developed above, of several
known curvature computations. Statement \emph{(1)} recovers Arnold's classical formula from \cite{arnold_sur_1966}, but derives it by a different method, namely through the embedding \(\SDifftor \hookrightarrow \Difftor\). This approach also appears in \cite{misiolek_stability_1993};
see also \cite{arnold_topological_2021}. Statements \emph{(2)} and \emph{(3)}
compute the curvature of the horizontal lift and the resulting Wasserstein
sectional curvatures in the Fourier basis on the flat torus, giving an
explicit torus version of the Wasserstein curvature calculation
in~\cite{otto_geometry_2001,lott_geometric_2007}.
\subsection{Proof of the curvature formulas}

\begin{lemma}\label{lem:vecttor.bracket}
At $\mu=dx^2$, the bracket of the frame elements in
$\VecttorC\ltimes\Denstor$ is
\[
[X_{\mbf k},X_{\mbf l}](dx^2)
=-\frac{1}{4\pi}\sum_{n_3\in\mbb Z_2}
\left((2l+k)\cdot b^{\mbf k}(b^{\mbf l}\cdot b^{(-k-l,n_3)})
-(2k+l)\cdot b^{\mbf l}(b^{\mbf k}\cdot b^{(-k-l,n_3)})\right)X_{(k+l,n_3)}.
\]
\end{lemma}

\begin{proof}[Proof of Lemma~\ref{lem:vecttor.bracket}]
Write $\mu=dx^2/\rho(x)^2$.  The commutator of the vector fields
$X_{\mbf k}(\mu)$ and $X_{\mbf l}(\mu)$ is
\[
[X_{\mbf k},X_{\mbf l}]_{\Vect}
=\frac{\rho e^{i(k+l)\cdot x}}{4\pi^2}
\left((b^{\mbf k}\times b^{\mbf l})\sgrad\rho
-i\rho\bigl((l\cdot b^{\mbf k})a^{\mbf l}
-(k\cdot b^{\mbf l})a^{\mbf k}\bigr)\right).
\]
Moreover, $\anchor X(\mu)=-\Lie_X\mu$, and hence
\[
\anchor X_{\mbf k}(\mu)
=\frac{e^{ik\cdot x}}{2\pi i\rho^2}
\left(\grad\rho\cdot a^{\mbf k}
-i\rho(k\cdot b^{\mbf k})\right)dx^2.
\]
To compute $\anchor X_{\mbf k}(\mu)\cdot X_{\mbf l}(\mu)$, consider the
curve $\lambda(t)=\mu+t\anchor X_{\mbf k}(\mu)$ in the complexified affine
space of densities of total mass $4\pi^2$.  Since
$\rho=(\mu/dx^2)^{-1/2}$, differentiation along this curve gives
\[
\anchor X_{\mbf k}(\mu)\cdot X_{\mbf l}(\mu)
=\frac{d}{dt}\Big|_0X_{\mbf l}
 \left(\mu+t\anchor X_{\mbf k}(\mu)\right)
=\frac{\rho e^{i(k+l)\cdot x}}{8\pi^2}
 \left(\grad\rho\cdot a^{\mbf k}
 -i\rho(k\cdot b^{\mbf k})\right)a^{\mbf l}.
\]
Combining these expressions in the action-algebroid bracket yields
\begin{equation*}
[X_{\mbf k},X_{\mbf l}](\mu)
=\frac{\rho e^{i(k+l)\cdot x}}{4\pi^2}\Big(
 \frac12(b^{\mbf k}\times b^{\mbf l})\sgrad\rho -i\rho\left(((l+k/2)\cdot b^{\mbf k})a^{\mbf l} -((k+l/2)\cdot b^{\mbf l})a^{\mbf k}\right)\Big).
\end{equation*}
At $\mu=dx^2$, only the frequency $k+l$ occurs.  Thus, for
$n_3\in\mbb Z_2$, the corresponding structure function is
\begingroup
\medmuskip=3mu\relax
\[
c_{\mbf k\,\mbf l}^{\,(k+l,n_3)}
=\ip{[X_{\mbf k},X_{\mbf l}](dx^2)}{X_{(-k-l,n_3)}}
=-\tfrac{1}{4\pi}\left((2l+k)\cdot b^{\mbf k}(b^{\mbf l}\cdot b^{(-k-l,n_3)})
-(2k+l)\cdot b^{\mbf l}(b^{\mbf k}\cdot b^{(-k-l,n_3)})\right).
\]
\endgroup
Summing over $n_3$ proves the lemma.
\end{proof}

\begin{proof}[Proof of Theorem~\ref{th:vecttor.hor.vert.curv}]
We first compute the second fundamental form $\II$ of the inclusion
$(\SDifftor)_{dx^2}\hookrightarrow\Difftor\ltimes\set{dx^2}$.
By Theorem~\ref{th:curvature.maps.between.derivations}\emph{(4)}, for
$V_k,V_l\in(\SVecttor_{\mbb C})_{dx^2}$,
\[
\II_{V_k}V_l
=\frac12\anchor\left(
 \ad^\madj_{V_k}V_l+\ad^\madj_{V_l}V_k\right).
\]
In terms of the structure functions of Lemma~\ref{lem:vecttor.bracket},
\[
\resizebox{\textwidth}{!}{$\displaystyle
\II_{V_k}V_l
=\frac12\sum_{n\in\mbb Z^2}
\left(
\ip{\ad^\madj_{V_k}V_l}{\gamma(S_{-n})}
+\ip{\ad^\madj_{V_l}V_k}{\gamma(S_{-n})}
\right)S_n
=-\frac12\left(
c^{\,(-l,1)}_{(k,1)(-k-l,0)}
+c^{\,(-k,1)}_{(l,1)(-k-l,0)}
\right)S_{k+l}.
$}
\]
Thus
\[
\II_{V_k}V_l
=-\frac{(k\times l)^2}
 {2\pi\norm{k}\norm{l}\norm{k+l}}S_{k+l}.
\]
Since $\nabla$ is flat, Corollary~\ref{cor:algebroid.Oneill.formulas} gives,
when $k+l+m+n=0$,
\[
\begin{aligned}
\ip{\curv^{\widehat\nabla}(V_k,V_l)V_m}{V_n}
&=\ip{\II_{V_k}V_m}{\II_{V_l}V_n}
-\ip{\II_{V_l}V_m}{\II_{V_k}V_n}\\
&=\frac{1}{4\pi^2\norm{k}\norm{l}\norm{m}\norm{n}}
\left(
\frac{(k\times m)^2(l\times n)^2}
 {\norm{k+m}\norm{l+n}}
-\frac{(l\times m)^2(k\times n)^2}
 {\norm{l+m}\norm{k+n}}
\right).
\end{aligned}
\]
Here $\norm{l+n}=\norm{k+m}$ and
$\norm{k+n}=\norm{l+m}$, which gives the formula in statement~\emph{(1)}.
If $k+l+m+n\neq0$, Fourier orthogonality makes both inner products vanish.

By equation~\eqref{eq:gamma.curvature.form.form},
\[
\curv^\gamma(S_k,S_l)
=-\iota\circ\omega[\gamma(S_k),\gamma(S_l)].
\]
Here the bracket must be computed using horizontal sections.  Although
$X_{(k,0)}(dx^2)=\gamma(S_k)$, the section $X_{(k,0)}(\mu)$ is generally
not horizontal away from $dx^2$.  We therefore extend $\gamma(S_k)$ and
$\gamma(S_l)$ by the density-independent gradient fields
\[
\frac{e^{ik\cdot x}}{2\pi i}a^{(k,0)}
\quad\text{and}\quad
\frac{e^{il\cdot x}}{2\pi i}a^{(l,0)}.
\]
Since these sections are independent of the density, their algebroid
bracket equals their vector-field commutator:
\[
\left[\frac{e^{ik\cdot x}}{2\pi i}a^{(k,0)},
      \frac{e^{il\cdot x}}{2\pi i}a^{(l,0)}\right]_{\Vect}
=-\frac{i(k\cdot l)e^{i(k+l)\cdot x}}
 {4\pi^2\norm{k}\norm{l}}
 \bigl((l_1-k_1)\p_{x_1}+(l_2-k_2)\p_{x_2}\bigr).
\]
For $k+l\neq0$, taking minus its vertical projection at $dx^2$ gives
\[
\curv^\gamma(S_k,S_l)
=\frac{k\cdot l}{2\pi\norm{k}\norm{l}}
 \bigl((l-k)\cdot b^{(-k-l,1)}\bigr)V_{k+l}
=\frac{(k\times l)(k\cdot l)}
 {\pi\norm{k}\norm{l}\norm{k+l}}V_{k+l}.
\]
For $l=-k$, the commutator is already a constant vertical vector field,
and hence
\[
\curv^\gamma(S_k,S_{-k})
=\frac{i}{2\pi^2}
 \bigl(k_1\p_{x_1}+k_2\p_{x_2}\bigr).
\]
This is nonzero for $k\neq0$, proving that the horizontal distribution is
not integrable and completing statement~\emph{(2)}.

We first apply Corollary~\ref{cor:algebroid.Oneill.formulas} to the basis
elements $\set{S_k}$ of the complexified tangent space
$T_{dx^2}\Denstor\otimes\mbb C$. Since $\nabla$ is flat,
\[
\ip{\curv^{\widetilde\nabla}(S_k,S_l)S_m}{S_n}
=\frac14\Big(
\ip{\curv^\gamma_{S_k}S_m}{\curv^\gamma_{S_l}S_n}
-\ip{\curv^\gamma_{S_l}S_m}{\curv^\gamma_{S_k}S_n} +2\ip{\curv^\gamma_{S_k}S_l}{\curv^\gamma_{S_m}S_n}\Big),
\]
which vanishes for $k+l+m+n\neq0$ by Fourier orthogonality.

For distinct $k,l\in\mbb Z^2_+$ and any pair
\[
(\widetilde S_k,\widetilde S_l)\in\set{
(S_k^{\Re},S_l^{\Re}),\,
(S_k^{\Re},S_l^{\Im}),\,
(S_k^{\Im},S_l^{\Re}),\,
(S_k^{\Im},S_l^{\Im})},
\]
expanding the real modes in equation~\eqref{eq:real.density.modes} gives
\[
\begin{aligned}
C^{\widetilde\nabla}(\widetilde S_k,\widetilde S_l)
&=\frac12\Big(
\ip{\curv^{\widetilde\nabla}(S_k,S_l)S_{-k}}{S_{-l}}
+\ip{\curv^{\widetilde\nabla}(S_k,S_{-l})S_{-k}}{S_l}\Big)\\
&=\frac38\Big(
\ip{\curv^\gamma_{S_k}S_l}{\curv^\gamma_{S_{-k}}S_{-l}}
+\ip{\curv^\gamma_{S_k}S_{-l}}{\curv^\gamma_{S_{-k}}S_l}\Big).
\end{aligned}
\]
Here the terms
$\ip{\curv^\gamma_{S_k}S_{-k}}{\curv^\gamma_{S_l}S_{-l}}$
cancel by skew-symmetry. Using statement~\emph{(2)}, we obtain
\[
C^{\widetilde\nabla}(\widetilde S_k,\widetilde S_l)
=\frac{3(k\times l)^2(k\cdot l)^2}
 {8\pi^2\norm{k}^2\norm{l}^2}
 \left(\frac{1}{\norm{k+l}^2}+\frac{1}{\norm{k-l}^2}\right).
\]
For the two real modes of one frequency, skew-symmetry instead yields
\[
\curv^\gamma(S_k^{\Re},S_k^{\Im})
=i\curv^\gamma(S_k,S_{-k})
=-\frac{1}{2\pi^2}
 \bigl(k_1\p_{x_1}+k_2\p_{x_2}\bigr).
\]
Its squared $L^2(dx^2)$ norm is $\norm{k}^2/\pi^2$, since
$\int_\torus dx^2=4\pi^2$.  Thus
\[
C^{\widetilde\nabla}(S_k^{\Re},S_k^{\Im})
=\frac{3\norm{k}^2}{4\pi^2},
\]
as asserted. More generally, the same O'Neill identity expresses the
sectional curvature on any real orthonormal plane as three quarters of the
squared norm of its horizontal curvature. This proves non-negativity at
every $\mu\in\Denstor$, finishing the proof of the theorem.
\end{proof}

\section{Proofs of the main curvature identities}\label{sec:proofs.of.main.theorems}
\subsection{The sectional curvature formula}
\begin{proof}[Proof of Theorem~\ref{th:2-1-2-3-4}]
By right-invariance, it suffices to compute
$\ip{\curv^\nabla(X,Y)X}{Y}(y)$ for local orthonormal sections $X,Y$
as in the theorem. Using the notation of the theorem, the Koszul formula
and Lemma~\ref{lem:levi.civita.formula} give
\[
\nabla_XY=\alpha+\delta,\qquad
\nabla_YX=-\alpha+\delta,\qquad
\nabla_XX=2B_X,\qquad
\nabla_YY=2B_Y.
\]
Metric compatibility, torsion-freeness, and
Definition~\ref{def:metric.adjoint} yield the three terms
\[
\begin{gathered}
\ip{\nabla_{[X,Y]}X}{Y}=\ip{\ad_XY}{\ad^\madj_XY-\nabla_XY}=-2\ip{\alpha}{\alpha}+2\ip{\alpha}{\beta},\\
-\ip{\nabla_X\nabla_YX}{Y}=-\anchor X\cdot\ip{\nabla_YX}{Y}+\ip{\nabla_YX}{\nabla_XY}=2\anchor X\cdot\ip{\alpha}{Y}+\ip{\delta}{\delta}-\ip{\alpha}{\alpha},\\
\ip{\nabla_Y\nabla_XX}{Y}=\anchor Y\cdot\ip{\nabla_XX}{Y}-\ip{\nabla_XX}{\nabla_YY}=-2\anchor Y\cdot\ip{\alpha}{X}-4\ip{B_X}{B_Y}.
\end{gathered}
\]
Combining these expressions in
\[
\ip{\curv^\nabla(X,Y)X}{Y}
=\ip{\nabla_{[X,Y]}X}{Y}
-\ip{\nabla_X\nabla_YX}{Y}
+\ip{\nabla_Y\nabla_XX}{Y}
\]
gives the stated formula on evaluating at $y$. Its independence of the
local orthonormal extensions follows from the tensoriality of
$\curv^\nabla$.
\end{proof}

\subsection{The block decomposition}
\begin{proof}[Proof of Theorem~\ref{th:block.diag.derivation}]
For $D\in\Gamma\Der_{\so}(\Aroid)$, $f\in C^\infty(\Base)$,
$S\in\Gamma T\Base$, and $V\in\Gamma\vertbundle$, the derivation rule gives
\[
\begin{aligned}
g(D)(fS) =& \anchor \Big( (\anchor_{\Der}D\cdot  f) \gamma(S) +  f D \circ  \gamma(S) \Big)
=(\anchor_{\Der}D\cdot f)S + f\,g(D)S ,\\
\widehat g(D)(fV) =& \omega \Big( (\anchor_{\Der}D\cdot  f) \iota(V) +  f D \circ  \iota(V) \Big)
=(\anchor_{\Der}D\cdot f)V + f\,\widehat g(D)V.
\end{aligned}
\]
Both diagonal blocks are therefore derivations with symbol
$\anchor_{\Der}D$.  Since $D$ is metric and $\gamma,\iota$ are
isometries onto the horizontal and vertical bundles, respectively,
\[
\begin{aligned}
\anchor_{\Der}D\cdot\ip{S_1}{S_2}_{T\Base}
&=\ip{g(D)S_1}{S_2}_{T\Base}
 +\ip{S_1}{g(D)S_2}_{T\Base},\\
\anchor_{\Der}D\cdot\ip{V_1}{V_2}_{\vertbundle}
&=\ip{\widehat g(D)V_1}{V_2}_{\vertbundle}
 +\ip{V_1}{\widehat g(D)V_2}_{\vertbundle}.
\end{aligned}
\]
Thus $g$ and $\widehat g$ are anchor-preserving maps into the respective
metric derivation algebroids.

The connection $g\circ\nabla\circ\gamma$ is metric, and its torsion
vanishes because, for $S_1,S_2\in\Gamma T\Base$,
\[
(g\circ\nabla\circ\gamma(S_1))S_2
-(g\circ\nabla\circ\gamma(S_2)) S_1
=\anchor\circ [\gamma(S_1),\gamma(S_2)]_{\Aroid}
=[S_1,S_2]_{T\Base}.
\]
It is therefore the Levi-Civita connection $\widetilde\nabla$ on
$T\Base$.  Likewise, $\widehat g\circ\nabla\circ\iota$ is metric and
\[
(\widehat g\circ\nabla\circ\iota(V_1))V_2
-(\widehat g\circ\nabla\circ\iota(V_2)) V_1
=\omega \circ [\iota(V_1),\iota(V_2)]_{\Aroid}
=[V_1,V_2]_{\vertbundle},
\]
so it is the Levi-Civita connection $\widehat\nabla$ on $\vertbundle$.
This proves (1) and (2).

For the off-diagonal blocks, the same derivation rule and the identities
$\anchor\circ\iota=0$ and $\omega\circ\gamma=0$ give
\[
(\anchor\circ D\circ\iota)(fV)
=f(\anchor\circ D\circ\iota)(V),
\qquad
(\omega\circ D\circ\gamma)(fS)
=f(\omega\circ D\circ\gamma)(S),
\]
which proves (3).  Finally, metric compatibility applied to the
orthogonal pair $\gamma(S),\iota(V)$ yields
\[
\ip{\omega\circ D\circ\gamma(S)}{V}_{\vertbundle}
=\ip{D\circ \gamma(S)}{\iota(V)}_{\Aroid}
=-\ip{\gamma(S)}{D\circ \iota(V)}_{\Aroid}
=-\ip{S}{\anchor\circ D\circ\iota(V)}_{T\Base},
\]
which proves (4).
\end{proof}

\subsection{Curvature of the projection maps}
\begin{proof}[Proof of Theorem~\ref{th:curvature.maps.between.derivations}]
\begin{enumerate}
\item The brackets in the metric derivation algebroids are commutators.
For $D_1,D_2\in\Gamma\Der_{\so}(\Aroid)$,
Definition~\ref{def:curvature}, the formula
$g(D)=\anchor\circ D\circ\gamma$, and
$\gamma\circ\anchor+\iota\circ\omega=\Id_{\Aroid}$ give
\[
\begin{gathered}
\curv^g(D_1,D_2)=g([D_1,D_2])-[g(D_1),g(D_2)]=\anchor\circ\bigl(D_1\circ\iota\circ\omega\circ D_2-D_2\circ\iota\circ\omega\circ D_1\bigr)\circ\gamma\\
=(\anchor\circ D_1\circ\iota)\circ(\omega\circ D_2\circ\gamma)-(\anchor\circ D_2\circ\iota)\circ(\omega\circ D_1\circ\gamma)=2\Skew\bigl((\anchor\circ D_1\circ\iota)\circ(\omega\circ D_2\circ\gamma)\bigr).
\end{gathered}
\]
The last equality follows from the negative-adjoint relation between the
off-diagonal blocks in
Theorem~\ref{th:block.diag.derivation}.\ref{th:off.diag.negative.metric.adjoint}.

It remains to prove the vanishing characterization in~\emph{(1)}.  If
$\dim(\Base)\leq 1$, then $\so(T\Base)=0$; if
$\rank(\vertbundle)=0$, then $\omega=0$.  In either case,
$\curv^g\equiv0$.

Conversely, suppose that $\dim(\Base)>1$ and
$\rank(\vertbundle)>0$.  At any $x\in\Base$, choose orthonormal vectors
$S_1,S_2\in T_x\Base$ and a unit vector $V_0\in\vertbundle_x$, and
denote their metric-dual covectors by $S^1,S^2$, and $V^0$.  Relative to
the vector-bundle decomposition
\[
\Aroid_x\simeq\gamma(T_x\Base)\oplus\iota(\vertbundle_x),
\]
define $\ChEnd_1,\ChEnd_2\in\so(\Aroid)_x$ by
\[
\ChEnd_i
=\gamma(S_i)\otimes(V^0\circ\omega)
-\iota(V_0)\otimes(S^i\circ\anchor),
\qquad i=1,2,
\]
viewed as operators $\Aroid_x\to\Aroid_x$.  Then
$g(\ChEnd_1)=g(\ChEnd_2)=0$, whereas
\[
\curv^g(\ChEnd_1,\ChEnd_2)=2\Skew\bigl((S_1\otimes V^0)\circ(-V_0\otimes S^2)\bigr)=S_2\otimes S^1-S_1\otimes S^2\neq0.
\]

Hence $\curv^g\equiv0$ only if $\dim(\Base)\leq1$ or
$\rank(\vertbundle)=0$.
\item 
For $S_1,S_2\in\Gamma T\Base$, put
$D_i=\nabla_{\gamma(S_i)}$, $i=1,2$.
By Remark~\ref{rem:horizontal.to.vertical.second.fundamental} and the
negative-adjoint relation in
Theorem~\ref{th:block.diag.derivation}.\ref{th:off.diag.negative.metric.adjoint},
\[
\omega\circ D_i\circ\gamma=-\frac12\curv^\gamma_{S_i}:T\Base\longrightarrow\vertbundle,\qquad\anchor\circ D_i\circ\iota=\frac12(\curv^\gamma_{S_i})^\trans:\vertbundle\longrightarrow T\Base.
\]
Substituting into \emph{(1)} gives
\[
\curv^g(\nabla_{\gamma(S_1)},\nabla_{\gamma(S_2)}) =-\frac12\Skew\bigl((\curv^\gamma_{S_1})^\trans
                       \curv^\gamma_{S_2}\bigr) =\frac12\Skew\bigl((\curv^\gamma_{S_2})^\trans
                      \curv^\gamma_{S_1}\bigr).
\]
\item The same commutator calculation, with
$\widehat g(D)=\omega\circ D\circ\iota$, gives
\[
\curv^{\widehat g}(D_1,D_2)
=(\omega\circ D_1\circ\gamma)\circ(\anchor\circ D_2\circ\iota)
 -(\omega\circ D_2\circ\gamma)\circ(\anchor\circ D_1\circ\iota) =2\Skew\bigl((\omega\circ D_1\circ\gamma)
 \circ(\anchor\circ D_2\circ\iota)\bigr).
\]
If $\rank(\vertbundle)\leq1$, then $\so(\vertbundle)=0$; if
$\dim(\Base)=0$, then $\anchor=0$. In either case,
$\curv^{\widehat g}\equiv0$.

Conversely, suppose that $\rank(\vertbundle)>1$ and $\dim(\Base)>0$.
At $x\in\Base$, choose orthonormal vectors
$V_1,V_2\in\vertbundle_x$ and a unit vector $S_0\in T_x\Base$,
with metric-dual covectors $V^1,V^2$, and $S^0$. Relative to the
vector-bundle decomposition
\[
\Aroid_x\simeq\gamma(T_x\Base)\oplus\iota(\vertbundle_x),
\]
define the skew endomorphisms
\[
\ChEnd_i
=\iota(V_i)\otimes(S^0\circ\anchor)
-\gamma(S_0)\otimes(V^i\circ\omega),
\qquad i=1,2.
\]
Then $\widehat g(\ChEnd_1)=\widehat g(\ChEnd_2)=0$, whereas
\[
\curv^{\widehat g}(\ChEnd_1,\ChEnd_2)=2\Skew\bigl((V_1\otimes S^0)\circ(-S_0\otimes V^2)\bigr)=V_2\otimes V^1-V_1\otimes V^2\neq0.
\]
Thus $\curv^{\widehat g}\equiv0$ only if
$\rank(\vertbundle)\leq1$ or $\dim(\Base)=0$.
\item For $V_1,V_2\in\Gamma\vertbundle$, put
$D_i=\nabla_{\iota(V_i)}$, $i=1,2$.
Since $\anchor\circ\iota=0$, these are sections of $\so(\Aroid)$.
By Remark~\ref{rem:second.fundamental.form} and
Theorem~\ref{th:block.diag.derivation}.\ref{th:off.diag.negative.metric.adjoint},
\[
\anchor\circ D_i\circ\iota=\II_{V_i}:\vertbundle\longrightarrow T\Base,\qquad\omega\circ D_i\circ\gamma=-(\II_{V_i})^\trans:T\Base\longrightarrow\vertbundle.
\]
Substituting into \emph{(3)} gives
\[
\curv^{\widehat g}(\nabla_{\iota(V_1)},\nabla_{\iota(V_2)})=-2\Skew\bigl((\II_{V_1})^\trans\II_{V_2}\bigr)=2\Skew\bigl((\II_{V_2})^\trans\II_{V_1}\bigr).
\]
\end{enumerate} 
\end{proof}

\section*{Acknowledgements}
I am deeply grateful to my supervisors, Irina Markina and Boris Khesin, for their generous support, encouragement, and guidance throughout this work.

\bibliography{bib.bib}

\end{document}